\documentclass{amsart}

\usepackage{kotex}
\usepackage{amsmath}
\usepackage{mathrsfs}
\usepackage{amsfonts}
\usepackage{amscd}
\usepackage{slashed}

\usepackage{amssymb}
\usepackage{graphicx}
\usepackage{caption}
\usepackage{subcaption}
\usepackage{dsfont}
\usepackage{color}
\usepackage{hyperref}
\usepackage{bbm}
\usepackage{a4wide}
\usepackage{sseq}
\usepackage{tikz-cd}
\usepackage[abs]{overpic} 

\newtheorem{theorem}{Theorem}[section]
\newtheorem{lemma}[theorem]{Lemma}
\newtheorem{proposition}[theorem]{Proposition}

\newtheorem{example}[theorem]{Example}

\theoremstyle{definition}
\newtheorem*{definition*}{Definition}
\newtheorem{definition}[theorem]{Definition}

\theoremstyle{remark}
\newtheorem*{remark*}{Remark}
\newtheorem{remark}[theorem]{Remark}

\theoremstyle{question}
\newtheorem*{question*}{Question}

\numberwithin{equation}{section}

\newcommand{\myproof}[2]{Proof of {#1} {#2}}

\begin{document}

\title[Categorical quantization on K\"ahler manifolds]{Categorical quantization on K\"ahler manifolds}

\author{YuTung Yau}
\address{Department of Mathematics, University of Michigan, Ann Arbor, MI, 48109, USA}
\email{ytyau@umich.edu}

\thanks{}

\maketitle

% \tableofcontents

\begin{abstract}
	Generalizing deformation quantizations with separation of variables of a K\"ahler manifold $M$, we adopt Fedosov's gluing argument to construct a category $\mathsf{DQ}$, enriched over sheaves of $\mathbb{C}[[\hbar]]$-modules on $M$, as a quantization of the category of Hermitian holomorphic vector bundles over $M$ with morphisms being smooth sections of hom-bundles.\par
	We then define quantizable morphisms among objects in $\mathsf{DQ}$, generalizing Chan-Leung-Li's notion \cite{ChaLeuLi2023} of quantizable functions. Upon evaluation of quantizable morphisms at $\hbar = \tfrac{\sqrt{-1}}{k}$, we obtain an enriched category $\mathsf{DQ}_{\operatorname{qu}, k}$. We show that, when $M$ is prequantizable, $\mathsf{DQ}_{\operatorname{qu}, k}$ is equivalent to the category $\mathsf{GQ}$ of holomorphic vector bundles over $M$ with morphisms being holomorphic differential operators, via a functor obtained from Bargmann-Fock actions.
\end{abstract}

\section{Introduction}
Let $(M, \omega)$ be a K\"ahler manifold. When $(M, \omega)$ is compact and admits a prequantum line bundle $L$, geometric quantization of $(M, \omega)$ yields the quantum Hilbert space $H^0(M, L^{\otimes k})$ for $k \in \mathbb{Z}^+$. A deformation quantization $(\mathcal{C}_M^\infty[[\hbar]], \star)$ of $(M, \omega)$, which is a non-commutative deformation of the sheaf $\mathcal{C}_M^\infty$ of $\mathbb{C}$-valued smooth functions on $M$, is specified by the asymptotic action of $\mathcal{C}^\infty(M, \mathbb{C})$ on $H^0(M, L^{\otimes k})$ via Toeplitz operators as $\hbar = \tfrac{\sqrt{-1}}{k} \to 0$ \cite{Cha2006, Sch2000}.\par
In \cite{ChaLeuLi2023}, Chan-Leung-Li introduced the notion of \emph{formal quantizable functions}, which form a dense subsheaf $\mathcal{C}_{M, \operatorname{qu}}^\infty$ of $(\mathcal{C}_M^\infty[[\hbar]], \star)$ and can be evaluated at $\hbar = \tfrac{\sqrt{-1}}{k}$ without convergence issues to obtain the sheaf $\mathcal{C}_{M, \operatorname{qu}, k}^\infty$ of \emph{level-}$k$ \emph{quantizable functions}. They constructed an action of $\mathcal{C}_{M, \operatorname{qu}, k}^\infty$ on the sheaf $\mathcal{L}^{\otimes k}$ of holomorphic sections of $L^{\otimes k}$ and showed that $\mathcal{C}_{M, \operatorname{qu}, k}^\infty$ is isomorphic to the sheaf $\mathcal{D}(\mathcal{L}^{\otimes k}, \mathcal{L}^{\otimes k})$ of holomorphic differential operators on $L^{\otimes k}$ via this action.\par
In this paper, we apply Chan-Leung-Li's sheaf-theoretic technique to quantize not only $(M, \omega)$ coupled with Hermitian holomorphic vector bundles $E_i$ over $M$, but also $\operatorname{End}(E_i)$-$\operatorname{End}(E_j)$ bimodules $\operatorname{Hom}(E_i, E_j)$. The whole quantum structure is encoded in a category $\mathsf{DQ}$, enriched over sheaves of $\mathbb{C}[[\hbar]]$-modules on $M$, whose objects are Hermitian holomorphic vector bundles over $M$. Analogously, \emph{formal quantizable morphisms} in $\mathsf{DQ}$ form an enriched subcategory $\mathsf{DQ}_{\operatorname{qu}}$ of $\mathsf{DQ}$, and we obtain an enriched category $\mathsf{DQ}_{\operatorname{qu}, k}$ from $\mathsf{DQ}_{\operatorname{qu}}$ by evaluation at $\hbar = \tfrac{\sqrt{-1}}{k}$.\par
While we regard $\mathsf{DQ}$ as a `categorification' of deformation quantizations of $(M, \omega)$, we also `categorify' the quantum Hilbert space $H^0(M, L^{\otimes k})$ - we consider the enriched category $\mathsf{GQ}$ of holomorphic vector bundles over $M$ with morphisms being holomorphic differential operators among those bundles. A `categorification' of the above sheaf action of non-formal deformation quantization on geometric quantization is encoded by an enriched functor
\begin{equation*}
	\mathscr{T}_k: \mathsf{DQ}_{\operatorname{qu}, k} \to \mathsf{GQ}
\end{equation*}
sending $E$ to $E \otimes L^{\otimes k}$, which is proved to be an equivalence of enriched categories in this paper.\par
We now state our first main result, which shows the existence of $\mathsf{DQ}$ as a `categorification' of deformation quantizations with separation of variables of $(M, \omega)$.
\begin{theorem}
	\label{Theorem 1.1}
	Let $(M, \omega)$ be a K\"ahler manifold. Then there exists a category $\mathsf{DQ}$, enriched over the monoidal category $\mathsf{Sh}_\hbar(M)$ of sheaves of $\mathbb{C}[[\hbar]]$-modules on $M$, defined as follows:
	\begin{itemize}
		\item objects in $\mathsf{DQ}$ are Hermitian holomorphic vector bundles over $M$;
		\item for any objects $E_1, E_2$ in $\mathsf{DQ}$ and open subset $U$ of $M$,
		\begin{equation*}
			\operatorname{Hom}_{\mathsf{DQ}} (E_1, E_2)(U) = \mathcal{C}^\infty(U, \operatorname{Hom}(E_1, E_2))[[\hbar]];
		\end{equation*}
		\item for any objects $E_1, E_2, E_3$ in $\mathsf{DQ}$, the composition
		\begin{equation*}
			\operatorname{Hom}_{\mathsf{DQ}} (E_2, E_3) \otimes_{\mathbb{C}[[\hbar]]} \operatorname{Hom}_{\mathsf{DQ}} (E_1, E_2) \to \operatorname{Hom}_{\mathsf{DQ}} (E_1, E_3)
		\end{equation*}
		is given by $\star$ defined as in (\ref{Equation 4.9}).
	\end{itemize}
	Moreover, for any objects $E_1, E_2$ in $\mathsf{DQ}$, open subset $U$ of $M$, $\phi \in \mathcal{C}^\infty(U, \operatorname{Hom}(E_1, E_2))[[\hbar]]$ and $\psi \in \mathcal{C}^\infty(U, \operatorname{Hom}(E_2, E_3))[[\hbar]]$, the following conditions are satisfied:
	\begin{enumerate}
		\item (\emph{classical limit}) if $\phi \in \mathcal{C}^\infty(U, \operatorname{Hom}(E_1, E_2))$ and $\psi \in \mathcal{C}^\infty(U, \operatorname{Hom}(E_2, E_3))$, then
		\begin{equation*}
			\psi \star \phi = \psi \phi \pmod \hbar;
		\end{equation*}
		\item (\emph{semi-classical limit}) if $\phi \in \mathcal{C}^\infty(U, \operatorname{Hom}(E_1, E_2))$ and $f \in \mathcal{C}^\infty(U, \mathbb{C})$, then
		\begin{equation*}
			(f \operatorname{Id}_{E_2}) \star \phi - \phi \star (f\operatorname{Id}_{E_1}) = \hbar \{f, \phi\} \pmod {\hbar^2},
		\end{equation*}
		where $\{f, \phi\} := \nabla_{X_f}^{\operatorname{Hom}(E_1, E_2)} \phi$ and $X_f$ is the $\omega$-Hamiltonian vector field of $f$;
		\item (\emph{separation of variables}) if $\psi$ is holomorphic or $\phi$ is anti-holomorphic, then $\psi \star \phi = \psi \phi$;
		\item (\emph{degree preserving property}) if $\phi, \psi$ are formal quantizable of degrees $r_1, r_2$ respectively, then $\psi \star \psi$ is formal quantizable of degree $r_1 + r_2$.
	\end{enumerate}
\end{theorem}

For a Hermitian holomorphic vector bundle $E$ over $M$, the endomorphism sheaf $\operatorname{Hom}_{\mathsf{DQ}}(E, E)$ together with its algebra structure given by $\star$ forms a deformation quantization with separation of variables of $\operatorname{End}(E)$, which were studied in \cite{Kar2014}. In particular, when $E$ is of rank $1$, $\operatorname{Hom}_{\mathsf{DQ}}(E, E)$ forms a deformation quantization of $(M, \omega)$ under the canonical isomorphism
\begin{equation*}
	\mathcal{C}^\infty(M, \operatorname{Hom}(E, E)) \cong \mathcal{C}^\infty(M, \mathbb{C}).
\end{equation*}
A key technique in the proof of Theorem \ref{Theorem 1.1} is Fedosov's gluing argument, which is used by Chan-Leung-Li in \cite{ChaLeuLi2023} as well. Now, assume $(M, \omega)$ is equipped with a prequantum line bundle $L$. Another technique to yield a similar categorical structure is the use of Toeplitz operators \cite{MaMar2012, AdaIshKan2023} (which requires that $M$ is compact). The former technique has an advantage over the latter - it is much more straightforward to see that the category $\mathsf{DQ}$ obtained is enriched over $\mathsf{Sh}_\hbar(M)$.\par
Passing from formal to non-formal quantization by evaluation at $\hbar = \tfrac{\sqrt{-1}}{k}$ (for $k \in \mathbb{Z}^+$), we then consider the enriched category $\mathsf{DQ}_{\operatorname{qu}, k}$ having the same objects as $\mathsf{DQ}$. For any objects $E_1, E_2$ in $\mathsf{DQ}_{\operatorname{qu}, k}$, the sheaf $\operatorname{Hom}_{\mathsf{DQ}_{\operatorname{qu}, k}}(E_1, E_2)$ is indeed a filtered left $\mathcal{O}_M$-module, where $\mathcal{O}_M$ is the sheaf of holomorphic functions on $M$. By Fedosov's gluing argument again, we glue the fibrewise Bargmann-Fock action to obtain a morphism of sheaves
\begin{equation*}
	\circledast_k: \operatorname{Hom}_{\mathsf{DQ}_{\operatorname{qu}, k}}(E_1, E_2) \times \mathcal{E}_1 \otimes \mathcal{L}^{\otimes k} \to \mathcal{E}_2 \otimes \mathcal{L}^{\otimes k},
\end{equation*}
where $\mathcal{E}_i$ is the sheaf of holomorphic sections of $E_i$ for $i = 1, 2$.\par
Our second main result states that $\circledast_k$ induces an equivalence $\mathscr{T}_k: \mathsf{DQ}_{\operatorname{qu}, k} \to \mathsf{GQ}$ of enriched categories, `categorifying' the sheaf action of non-formal quantizable functions on holomorphic sections of $L^{\otimes k}$ constructed by Chan-Leung-Li in \cite{ChaLeuLi2023}.

\begin{theorem}
	\label{Theorem 1.2}
	Let $(M, \omega)$ be a prequantizable K\"ahler manifold with a prequantum line bundle $L$ and $k \in \mathbb{Z}^+$. Then there exists an enriched functor 
	\begin{equation*}
		\mathscr{T}_k: \mathsf{DQ}_{\operatorname{qu}, k} \to \mathsf{GQ}
	\end{equation*}
	such that
	\begin{enumerate}
		\item for any object $E$ in $\mathsf{DQ}_{\operatorname{qu}, k}$, $\mathscr{T}_k(E) = E \otimes L^{\otimes k}$;
		\item for any objects $E_1, E_2$ in $\mathsf{DQ}_{\operatorname{qu}, k}$,
		\begin{equation*}
			\mathscr{T}_k: \operatorname{Hom}_{\mathsf{DQ}_{\operatorname{qu}, k}}(E_1, E_2) \to \operatorname{Hom}_\mathsf{GQ}(\mathscr{T}_k(E_1), \mathscr{T}_k(E_2))
		\end{equation*}
		is an isomorphism of filtered left $\mathcal{O}_M$-modules;
		\item $\mathscr{T}_k$ yields an equivalence of categories enriched over the monoidal category $\mathsf{Sh}(M)$ of sheaves of $\mathbb{C}$-vector spaces on $M$.
	\end{enumerate}
\end{theorem}

In \cite{Yau2024}, the author of this paper investigated how Chan-Leung-Li's work \cite{ChaLeuLi2023} mathematically realizes a certain part of the categorical structure for (rank-$1$) \emph{coisotropic A-branes} on a sufficiently small complexification $X$ of $M$, following Gukov-Witten's physical proposal for quantization \cite{GukWit2009}. This categorical structure was first expected by Kapustin-Orlov \cite{KapOrl2003} for the sake of the Homological Mirror Symmetry Conjecture \cite{Kon1995}. One of the potential applications of our main results is that the category $\mathsf{DQ}_{\operatorname{qu}, k}$, the functor $\mathscr{T}_k: \mathsf{DQ}_{\operatorname{qu}, k} \to \mathsf{GQ}$ and its \emph{transposed functor} $\mathscr{T}_k^{\operatorname{t}}: \mathsf{DQ}_{\operatorname{qu}, k} \to  \mathsf{GQ}^{\operatorname{op}}$ (see Remark \ref{Remark 5.10}) might serve as a mathematical realization of a larger part of the above categorical structure which involves \emph{higher rank coisotropic A-branes} \cite{Her2012}. It will be studied in the future.\par
Throughout this paper, except in Appendix \ref{Appendix B}, assume that $(M, \omega)$ is a K\"ahler manifold. A list of notations are adopted as below:
\begin{itemize}
	\item $\hbar$ is a formal variable over $\mathbb{C}$.
	\item We write $T_\mathbb{C} = TM_\mathbb{C}$, $T = T^{1, 0}M$ and $\overline{T} = T^{0, 1}M$.
	\item Denote by $\mathsf{Sh}(M)$ (resp. $\mathsf{Sh}_\hbar(M)$) the monoidal category of sheaves of $\mathbb{C}$-vector spaces (resp. $\mathbb{C}[[\hbar]]$-modules) on $M$ equipped with the tensor product $\otimes_\mathbb{C}$ (resp. $\otimes_{\mathbb{C}[[\hbar]]}$).
	\item Denote by $\mathcal{O}_M$ (resp. $\mathcal{C}_M^\infty$) the sheaf of holomorphic functions (resp. complex-valued smooth functions) on $M$.
	\item Denote by $\nabla$ the Levi-Civita connection of the $(M, \omega)$. 
	\item For each holomorphic vector bundle $E$ over $M$, denote by $\mathcal{E}$ the sheaf of holomorphic sections of $E$ and by $JE$ the holomorphic jet bundle of $E$.
	\item For any two holomorphic vector bundles $E_1$, $E_2$ over $M$, denote by $\mathcal{D}(\mathcal{E}_1, \mathcal{E}_2)$ the sheaf of holomorphic differential operators from $E_1$ to $E_2$. 
	\item For a Hermitian holomorphic vector bundle $E$ over $M$, denote by $\nabla^E$ its Chern connection.
	\item For a complex vector bundle $E$ over $M$, we denote by $\operatorname{Sym} E$ (resp. $\widehat{\operatorname{Sym}} E$) the symmetric algebra bundle (resp. completed symmetric algebra bundle) of $E$. We also define
	\begin{center}
		\begin{tabular}{ll}
			$\mathcal{W} = \widehat{\operatorname{Sym}} T_\mathbb{C}^\vee[[\hbar]],$ & $\mathcal{W}_{\operatorname{cl}} = \widehat{\operatorname{Sym}} T_\mathbb{C}^\vee,$\\
			$\underline{\mathcal{W}} = \widehat{\operatorname{Sym}} T^\vee \otimes \operatorname{Sym} \overline{T}^\vee [\hbar],$ & $\underline{\mathcal{W}}_{\operatorname{cl}} = \widehat{\operatorname{Sym}} T^\vee \otimes \operatorname{Sym} \overline{T}^\vee,$\\
			$\mathcal{W}^{1, 0} = \widehat{\operatorname{Sym}} T^\vee [[\hbar]],$ & $\mathcal{W}_{\operatorname{cl}}^{1, 0} = \widehat{\operatorname{Sym}} T^\vee,$\\
			$\mathcal{W}^{0, 1} = \widehat{\operatorname{Sym}} \overline{T}^\vee [[\hbar]],$ & $\mathcal{W}_{\operatorname{cl}}^{0, 1} = \widehat{\operatorname{Sym}} \overline{T}^\vee.$
		\end{tabular}
	\end{center}
\end{itemize}

\section{A categorification of the sheaf of smooth functions}
We will define an enriched category $\mathsf{C}$ associated with the K\"ahler manifold $(M, \omega)$ in Subsection \ref{Subsection 2.1}, and introduce covariantized Poisson brackets $\{\quad, \quad\}$ in Subsection \ref{Subsection 2.2}. The pair $(\mathsf{C}, \{\quad, \quad\})$ is to be `quantized' to an enriched category $\mathsf{DQ}$ in Theorem \ref{Theorem 1.1}. In Subsection \ref{Subsection 2.3}, we will define deformation quantization with separation of variables of $(\mathsf{C}, \{\quad, \quad\})$. We will end this section by a discussion on tensor products with Hermitian holomorphic line bundles as auto-equivalences of $\mathsf{C}$ and as transformations of deformation quantizations of $(\mathsf{C}, \{\quad, \quad\})$ in Subsection \ref{Subsection 2.4}.

\subsection{The classical category of Hermitian holomorphic vector bundles}
\label{Subsection 2.1}
\quad\par
The enriched category to be `quantized' in this paper is defined as follows.

\begin{definition}
	The \emph{classical category of Hermitian holomorphic vector bundles over} $M$, denoted by $\mathsf{C}$, is the category enriched over $\mathsf{Sh}(M)$ given as follows:
	\begin{itemize}
		\item objects in $\mathsf{C}$ are Hermitian holomorphic vector bundles over $M$;
		\item for any objects $E_1, E_2$ in $\mathsf{C}$ and open subset $U$ of $M$,
		\begin{equation*}
			\operatorname{Hom}_{\mathsf{C}} (E_1, E_2)(U) = \mathcal{C}^\infty(U, \operatorname{Hom}(E_1, E_2));
		\end{equation*}
		\item for any objects $E_1, E_2, E_3$ in $\mathsf{C}$, the composition
		\begin{equation*}
			\operatorname{Hom}_{\mathsf{C}} (E_2, E_3) \otimes_\mathbb{C} \operatorname{Hom}_{\mathsf{C}} (E_1, E_2) \to \operatorname{Hom}_{\mathsf{C}} (E_1, E_3)
		\end{equation*}
		is the usual composition of smooth sections of hom-bundles.
	\end{itemize}
\end{definition}

If $E = M \times \mathbb{C}$ is the trivial Hermitian holomorphic line bundle over $M$, then the endomorphism sheaf $\operatorname{Hom}_\mathsf{C}(E, E)$ equipped with compositions is canonically isomorphic to the sheaf of algebras $\mathcal{C}_M^\infty$. Hence, we regard $\mathsf{C}$ as a `categorification' of the sheaf $\mathcal{C}_M^\infty$.

\begin{remark}
	For two objects $E_1, E_2$ in $\mathsf{C}$, as the hom-bundle $\operatorname{Hom}(E_1, E_2)$ is independent of the Hermitian holomorphic structures on $E_1, E_2$, so is $\operatorname{Hom}_\mathsf{C}(E_1, E_2)$. Indeed, if the underlying smooth complex vector bundles of $E_1, E_2$ are equal, then $E_1, E_2$ are isomorphic in $\mathsf{C}$.
\end{remark}

The reason why we require Hermitian holomorphic structures in the definition of $\mathsf{C}$ is as follows. As $(M, \omega)$ is a K\"ahler manifold, one can consider a deformation quantization $(\mathcal{C}_M^\infty[[\hbar]], \star)$ of $(M, \omega)$ \emph{with separation of variables} (in the sense of \cite{Kar1996}) - for an open subset $U$ of $M$ and $f, g \in \mathcal{C}^\infty(U, \mathbb{C})$, if $g$ is holomorphic or $f$ is anti-holomorphic, then
\begin{equation*}
	f \star g = fg.
\end{equation*}
One of the goals of this paper is to construct a `categorification' of deformation quantizations with separation of variables. As objects in $\mathsf{C}$ are Hermitian holomorphic, via Chern connections, holomorphicity and anti-holomorphicity of morphisms in $\mathsf{C}$ are well defined.

\subsection{Covariantized Poisson brackets}
\label{Subsection 2.2}
\quad\par
Recall that the symplectic form $\omega$ on $M$ induces a Poisson bracket $\{\quad, \quad\}$ on $\mathcal{C}^\infty(M, \mathbb{C})$. If we write $\omega = \omega_{ij} dx^i \wedge dx^j$ in local real coordinates $(x^1, ..., x^{2n})$ and let $(\omega^{ij})$ be the inverse of $(\omega_{ij})$, then the Poisson bracket is given by $\{f, g\} = \omega^{ij} \tfrac{\partial f}{\partial x^i} \tfrac{\partial g}{\partial x^j}$. It is generalized to a structure for each pair of Hermitian holomorphic vector bundles, known as a covariantized Poisson bracket \cite{AdaIshKan2023}.

\begin{definition}
	Let $E_1, E_2$ be Hermitian holomorphic vector bundles over $M$. The \emph{covariantized Poisson bracket} for the pair $(E_1, E_2)$ is defined as the $\mathbb{C}$-bilinear map
	\begin{align*}
		\{\quad, \quad\}: \mathcal{C}^\infty(M, E_1) \times \mathcal{C}^\infty(M, E_2) & \to \mathcal{C}^\infty(M, E_1 \otimes E_2)\\
		(s_1, s_2) & \mapsto \{s_1, s_2\} := \omega^{ij} \nabla_{\partial_{x^i}}^{E_1} s_1 \otimes \nabla_{\partial_{x^j}}^{E_2} s_2.
	\end{align*}
\end{definition}

Note that the hom-bundle between two Hermitian holomorphic vector bundles is itself a Hermitian holomorphic vector bundle. Compositions in $\mathsf{C}$ are compatible with covariantized Poisson brackets in the following sense.

\begin{proposition}
	Let $E, E_1, E_2, E_3$ be Hermitian holomorphic vector bundles over $M$. Then for all open subset $U$ of $M$, $s \in \mathcal{C}^\infty(U, E)$, $\phi \in \operatorname{Hom}_{\mathsf{C}} (E_1, E_2)(U)$ and $\psi \in \operatorname{Hom}_{\mathsf{C}} (E_2, E_3)(U)$,
	\begin{align*}
		\{ s, \psi \circ \phi \} = & \{ s, \psi \} \circ (\operatorname{Id}_E \otimes \phi) + (\operatorname{Id}_E \otimes \psi) \circ \{ s, \phi \},\\
		\{ \psi \circ \phi, s \} = & \{ \psi, s \} \circ (\phi \otimes \operatorname{Id}_E) + (\psi \otimes \operatorname{Id}_E) \circ \{ \phi, s \}.
	\end{align*}
\end{proposition}

If $f \in \mathcal{C}^\infty(M, \mathbb{C})$ is regarded as a smooth section of the trivial Hermitian holomorphic line bundle $M \times \mathbb{C}$ and $s \in \mathcal{C}^\infty(M, E)$ is a smooth section of a Hermitian holomorphic vector bundle $E$ over $M$, then $\{f, s\} = -\{s, f\} = \nabla_{X_f}^E s$, where $X_f$ is the $\omega$-Hamiltonian vector field of $f$. In particular, the corviantized Poisson bracket for the pair $(M \times \mathbb{C}, M \times \mathbb{C})$ reduces to the Poisson bracket on $M$ induced by $\omega$. Therefore, we can regard $\mathsf{C}$ equipped with the collection of covariantized Poisson brackets for all the pairs of the form $(M \times \mathbb{C}, \operatorname{Hom}(E_i, E_j))$ as a `categorification' of the sheaf of Poisson algebras $(\mathcal{C}_M^\infty, \{\quad, \quad\})$.

\subsection{Categorical deformation quantization with separation of variables}
\label{Subsection 2.3}
\quad\par
We now define deformation quantizations of $(\mathsf{C}, \{\quad, \quad\})$.

\begin{definition}
	\label{Definition 2.5}
	A \emph{deformation quantization of} $(\mathsf{C}, \{\quad, \quad\})$ is a category $\mathsf{C}_\hbar$ enriched over $\mathsf{Sh}_\hbar(M)$ which satisfies the following conditions:
	\begin{itemize}
		\item objects in $\mathsf{C}_\hbar$ are the same as those in $\mathsf{C}$;
		\item for any objects $E_1, E_2$ in $\mathsf{C}_\hbar$ and open subset $U$ of $M$,
		\begin{equation*}
			\operatorname{Hom}_{\mathsf{C}_\hbar} (E_1, E_2)(U) = \operatorname{Hom}_{\mathsf{C}} (E_1, E_2)(U)[[\hbar]];
		\end{equation*}
		\item (\emph{classical limit}) for any objects $E_1, E_2$ in $\mathsf{C}_\hbar$, open subset $U$ of $M$, $\phi \in \operatorname{Hom}_{\mathsf{C}} (E_1, E_2)(U)$ and $\psi \in \operatorname{Hom}_{\mathsf{C}} (E_2, E_3)(U)$,
		\begin{equation*}
			\psi \star \phi = \psi \phi \pmod \hbar,
		\end{equation*}
		where $\psi \star \phi$ (resp. $\psi \phi$) denotes the composition of $\psi$ and $\phi$ in $\mathsf{C}_\hbar$ (resp. $\mathsf{C}$);
		\item (\emph{semi-classical limit}) for any objects $E_1, E_2$ in $\mathsf{C}_\hbar$, open subset $U$ of $M$, $f \in \mathcal{C}^\infty(U, \mathbb{C})$ and $\phi \in \operatorname{Hom}_{\mathsf{C}} (E_1, E_2)(U)$,
		\begin{equation*}
			(f \operatorname{Id}_{E_2}) \star \phi - \phi \star (f\operatorname{Id}_{E_1}) = \hbar \{f, \phi\} \pmod {\hbar^2},
		\end{equation*}
		where $\{f, \phi\} := \nabla_{X_f}^{\operatorname{Hom}(E_1, E_2)} \phi$ and $X_f$ is the $\omega$-Hamiltonian vector field of $f$.
	\end{itemize}
\end{definition}

\begin{definition}
	\label{Definition 2.6}
	A deformation quantization $\mathsf{C}_\hbar$ of $(\mathsf{C}, \{ \quad, \quad \})$ is said to be \emph{with separation of variables} if for any objects $E_1, E_2$ in $\mathsf{C}_\hbar$, open subset $U$ of $M$, $\phi \in \operatorname{Hom}_{\mathsf{C}_\hbar}(E_1, E_2)(U)$ and $\psi \in \operatorname{Hom}_{\mathsf{C}_\hbar}(E_2, E_3)(U)$, if $\psi$ is holomorphic or $\phi$ is anti-holomorphic, then
	\begin{equation*}
		\psi \star \phi = \psi \phi,
	\end{equation*}
	where $\psi \star \phi$ (resp. $\psi \phi$) denotes the composition of $\psi$ and $\phi$ in $\mathsf{C}_\hbar$ (resp. $\mathsf{C}$).
\end{definition}

Let $\mathsf{C}_\hbar$ be a deformation quantization with separation of variables of $(\mathsf{C}, \{\quad, \quad\})$. Then for the trivial Hermitian holomorphic line bundle $E = M \times \mathbb{C}$ over $M$, the composition
\begin{equation*}
	\star: \operatorname{Hom}_{\mathsf{C}_\hbar}(E, E) \otimes_{\mathbb{C}[[\hbar]]} \operatorname{Hom}_{\mathsf{C}_\hbar}(E, E) \to \operatorname{Hom}_{\mathsf{C}_\hbar}(E, E)
\end{equation*}
is a star product with separation of variables on $(M, \omega)$. This suggests that $\mathsf{C}_\hbar$ is a `categorification' of deformation quantizations of $(M, \omega)$.

\subsection{Tensor product with a Hermitian holomorphic line bundle}
\label{Subsection 2.4}
\quad\par
Let $L$ be an object in $\mathsf{C}$ which is of rank $1$, i.e. a Hermitian holomorphic line bundle over $M$. For any objects $E_1, E_2$ in $\mathsf{C}$, we have a canonical isomorphism in $\mathsf{Sh}(M)$:
\begin{equation*}
	\operatorname{Hom}_\mathsf{C}(E_1, E_2) \cong \operatorname{Hom}_\mathsf{C}(E_1 \otimes L, E_2 \otimes L),
\end{equation*}
which identifies $\nabla^{\operatorname{Hom}(E_1, E_2)}$ with $\nabla^{\operatorname{Hom}(E_1 \otimes L, E_2 \otimes L)}$. In other words, the tensor product with $L$ forms an auto-equivalence of the enriched category $\mathsf{C}$:
\begin{equation}
	\otimes L: \mathsf{C} \to \mathsf{C},
\end{equation}
which preserves the covariantized Poisson brackets on the morphism sheaves in $\mathsf{C}$.\par
Now suppose $\mathsf{C}_\hbar$ is a deformation quantization with separation of variables of $(\mathsf{C}, \{\quad, \quad\})$ and $\star$ is the composition in $\mathsf{C}_\hbar$. Note that for any objects $E_1, E_2$ in $\mathsf{C}_\hbar$, there is a canonical isomorphism of sheaves of $\mathbb{C}[[\hbar]]$-modules on $M$:
\begin{equation*}
	\operatorname{Hom}_{\mathsf{C}_\hbar}(E_1, E_2) \cong \operatorname{Hom}_{\mathsf{C}_\hbar}(E_1 \otimes L, E_2 \otimes L).
\end{equation*}
These isomorphisms induce a category $\mathsf{C}_{\hbar, L}$ enriched over $\mathsf{Sh}_\hbar(M)$ which is completely determined by the following conditions:
\begin{itemize}
	\item objects and morphism sheaves in $\mathsf{C}_{\hbar, L}$ are the same as those in $\mathsf{C}_\hbar$; and
	\item the tensor product with $L$ forms an equivalence of enriched categories:
	\begin{equation}
		\otimes L: \mathsf{C}_\hbar \to \mathsf{C}_{\hbar, L}.
	\end{equation}
\end{itemize}
We can easily prove the following proposition.

\begin{proposition}
	Let $L$ be a Hermitian holomorphic line bundle over $M$. If $\mathsf{C}_\hbar$ is a deformation quantization with separation of variables of $(\mathsf{C}, \{\quad, \quad\})$, then so is $\mathsf{C}_{\hbar, L}$.
\end{proposition}

In particular, if $\mathsf{C}_\hbar$ is a deformation quantization with separation of variables of $(\mathsf{C}, \{\quad, \quad\})$ and $L$ is an object in $\mathsf{C}_\hbar$ of rank $1$, then $\operatorname{Hom}_{\mathsf{C}_\hbar}(L, L)$ equipped with the compositions is a deformation quantization with separation of variables of $(M, \omega)$.

\section{Formal geometry on K\"ahler manifolds via Kapranov's $L_\infty$-structures}
The goal of this section is to give a review on formal geometry on a K\"ahler manifold $(M, \omega)$ via Kapranov's $L_\infty$-structures so as to make this paper self-contained. Subsection \ref{Subsection 3.1} is a recollection of formal Dolbeault complexes. In Subsections \ref{Subsection 3.2} and \ref{Subsection 3.3}, we will introduce a class of flat connections called Kapranov's connections, and study quasi-isomorphisms among relevant cochain complexes which are important for both Sections \ref{Section 4} and \ref{Section 5}. In Subsection \ref{Subsection 3.4}, we will discuss the bundle of holomorphic differential operators between two holomorphic vector bundles.

\subsection{The formal Dolbeault complex of $\mathcal{W}_{\operatorname{cl}}^{1, 0} \otimes E$}
\label{Subsection 3.1}
\quad\par
Suppose that $E$ is a complex vector bundle over $M$. We can define three $\mathcal{C}^\infty(M, \mathbb{C})$-linear operators $\delta^{1, 0}, (\delta^{1, 0})^{-1}, \pi_{0, *}$ on $\Omega^*(M, \mathcal{W}_{\operatorname{cl}}^{1, 0} \otimes E)$, where $\mathcal{W}_{\operatorname{cl}}^{1, 0} = \widehat{\operatorname{Sym}} T^\vee$, as follows. For a local $\mathcal{W}_{\operatorname{cl}}^{1, 0} \otimes E$-valued form $s = (w^{\mu_1} \cdots w^{\mu_l} \otimes e) dz^{\alpha_1} \wedge \cdots \wedge dz^{\alpha_p} \wedge d\overline{z}^{\beta_1} \wedge \cdots \wedge d\overline{z}^{\beta_q}$,
\begin{align*}
	\delta^{1, 0} s = dz^\mu \wedge \frac{\partial s}{\partial w^\mu},\quad
	(\delta^{1, 0})^{-1} s =
	\begin{cases}
		\frac{1}{l+p} w^\mu \iota_{\partial_{z^\mu}} s & \text{ if } l + p > 0;\\
		0 & \text{ if } l + p = 0,
	\end{cases}
	\quad \pi_{0, *}(s) =
	\begin{cases}
		0 & \text{ if } l + p > 0;\\
		s & \text{ if } l + p = 0,
	\end{cases}
\end{align*}
Here, $(z^1, ..., z^n)$ are local complex coordinates on $M$, $e$ is a local section of $E$, and we denote by $w^\mu$ the covector $dz^\mu$ regarded as a local section of $\mathcal{W}_{\operatorname{cl}}^{1, 0}$. We also suppress the notations of symmetric product. Clearly, $(\Omega^*(M, \mathcal{W}_{\operatorname{cl}}^{1, 0} \otimes E), \delta^{1, 0})$ forms a cochain complex and we call it the \emph{formal Dolbeault complex} of $\mathcal{W}_{\operatorname{cl}}^{1, 0} \otimes E$. The following equality holds on $\Omega^*(M, \mathcal{W}_{\operatorname{cl}}^{1, 0} \otimes E)$:
\begin{equation}
	\label{Equation 3.1}
	\operatorname{Id} - \pi_{0, *} = \delta^{1, 0} \circ (\delta^{1, 0})^{-1} + (\delta^{1, 0})^{-1} \circ \delta^{1, 0},
\end{equation}
implying that the following inclusion is a quasi-isomorphism:
\begin{equation*}
	(\Omega^{0, *}(M, E), 0) \hookrightarrow (\Omega^*(M, \mathcal{W}_{\operatorname{cl}}^{1, 0} \otimes E), \delta^{1, 0}).
\end{equation*}
Now we state two fundamental lemmas which will be used in later (sub)sections.

\begin{lemma}
	\label{Lemma 3.1}
	Let $E$ be a holomorphic vector bundle over $M$. Then the equality
	\begin{equation*}
		\operatorname{Id} - \pi_{0, *} = (\delta^{1, 0} - \overline{\partial}) \circ (\delta^{1, 0})^{-1} + (\delta^{1, 0})^{-1} \circ (\delta^{1, 0} - \overline{\partial})
	\end{equation*}
	holds on $\Omega^*(M, \mathcal{W}_{\operatorname{cl}}^{1, 0} \otimes E)$. In particular, the inclusion
	\begin{equation*}
		(\Omega^{0, *}(M, E), \overline{\partial}) \hookrightarrow (\Omega^*(M, \mathcal{W}_{\operatorname{cl}}^{1, 0} \otimes E), \overline{\partial} - \delta^{1, 0})
	\end{equation*}
	is a quasi-isomorphism.
\end{lemma}

If $E$ is a Hermitian holomorphic vector bundle, we equip $\mathcal{W}_{\operatorname{cl}}^{1, 0} \otimes E$ with the connection induced by the Levi-Civita connection $\nabla$ of $M$ and the Chern connection $\nabla^E$. Here and in the sequel, we denote by $\nabla^{1, 0}$ the $(1, 0)$-part of any connection in the context and let $\widetilde{\nabla}^{1, 0} = (\delta^{1, 0})^{-1} \circ \nabla^{1, 0}$ whenever it is well defined.

\begin{lemma}
	\label{Lemma 3.2}
	Let $E$ be a Hermitian holomorphic vector bundle over $M$. Let $s \in \Omega^*(M, \mathcal{W}_{\operatorname{cl}}^{1, 0} \otimes E)$ and $\widetilde{s} = \sum_{r=0}^\infty (\widetilde{\nabla}^{1, 0})^r s$. If $\nabla^{1, 0} (\delta^{1, 0} s) = 0$, then
	\begin{equation*}
		(\nabla^{1, 0} - \delta^{1, 0}) \widetilde{s} = -\delta^{1, 0} s.
	\end{equation*}
\end{lemma}
\begin{proof}
	For any $r \in \mathbb{N}$, define $s_{(r)} = (\widetilde{\nabla}^{1, 0})^r s$. Observe that
	\begin{itemize}
		\item if $\nabla^{1, 0} (\delta^{1, 0} s_{(r)}) = 0$, then $\nabla^{1, 0} s_{(r)} = \delta^{1, 0} s_{(r+1)} + (\delta^{1, 0})^{-1} (\delta^{1, 0} (\nabla^{1, 0} s_{(r)})) = \delta^{1, 0} s_{(r+1)}$ by (\ref{Equation 3.1}) and the equalities $\delta^{1, 0} (\nabla^{1, 0} s_{(r)}) = -\nabla^{1, 0} (\delta^{1, 0} s_{(r)})$ and $\pi_{0, *} \circ \nabla^{1, 0} = 0$; and
		\item if $\nabla^{1, 0} s_{(r)} = \delta^{1, 0} s_{(r+1)}$, then $\nabla^{1, 0} ( \delta^{1, 0} s_{(r+1)}) = (\nabla^{1, 0})^2 s_{(r)} = 0$.
	\end{itemize}
	By hypothesis and induction, $(\nabla^{1, 0} - \delta^{1, 0})\widetilde{s} = \sum_{r=0}^\infty (\nabla^{1, 0}s_{(r)} - \delta^{1, 0}s_{(r+1)}) - \delta^{1, 0} s = -\delta^{1, 0} s$.
\end{proof}

\subsection{The Kapranov's connection $D_{\operatorname{Kap}}^E$ on $\mathcal{W}_{\operatorname{cl}}^{1, 0} \otimes E$}
\label{Subsection 3.2}
\quad\par
For any K\"ahler manifold $M$, Kapranov \cite{Kap1999} discovered an $L_\infty$-algebra structure encoded by higher covariant derivatives of the curvature $\nabla^2 \in \Omega^{1, 1}(M, \operatorname{Hom} (T^\vee, T^\vee))$ of $\nabla$ on $T^\vee$, namely
\begin{equation*}
	I_{(r)} := (\widetilde{\nabla}^{1, 0})^{r-2} (\delta^{1, 0})^{-1} \nabla^2 \in \Omega^{0, 1}(M, \operatorname{Hom}(T^\vee, \operatorname{Sym}^r T^\vee)), \quad \text{for } r \geq 2.
\end{equation*}
Furthermore, he showed that any Hermitian holomorphic vector bundle $E$ over $M$ induces an $L_\infty$-module structure encoded by higher covariant derivatives of the curvature $R^E \in \Omega^{1, 1}(M, \operatorname{End}(E))$ of its Chern connection $\nabla^E$, namely
\begin{equation*}
	I_{(r)}^E := ( \widetilde{\nabla}^{1, 0} )^{r-1} (\delta^{1, 0})^{-1} R^E \in \Omega^{0, 1}(M, \operatorname{Sym}^r T^\vee \otimes \operatorname{End}(E)), \quad \text{for } r \geq 1.
\end{equation*}
Define
\begin{equation*}
	I = \sum_{r=2}^\infty I_{(r)} \in \Omega^{0, 1}(M, \operatorname{Hom}(T^\vee, \mathcal{W}_{\operatorname{cl}}^{1, 0})).
\end{equation*}
It can be regarded as a $\mathcal{C}^\infty(M, \mathbb{C})$-linear map $\Omega^*(M, T^\vee) \to \Omega^*(M, \mathcal{W}_{\operatorname{cl}}^{1, 0})$. We extend it to a graded derivation $\Omega^*(M, \mathcal{W}_{\operatorname{cl}}^{1, 0}) \to \Omega^*(M, \mathcal{W}_{\operatorname{cl}}^{1, 0})$. Taking the tensor product with $\operatorname{Id}_E$, we obtain a $\mathcal{C}^\infty(M, \mathbb{C})$-linear map
\begin{equation*}
	\Omega^*(M, \mathcal{W}_{\operatorname{cl}}^{1, 0} \otimes E) \to \Omega^*(M, \mathcal{W}_{\operatorname{cl}}^{1, 0} \otimes E)
\end{equation*}
which is still denoted by $I$ by abuse of notations. On the other hand, define
\begin{equation*}
	I^E = \sum_{r=1}^\infty I_{(r)}^E \in \Omega^{0, 1}(M, \mathcal{W}_{\operatorname{cl}}^{1, 0} \otimes \operatorname{End}(E)).
\end{equation*}
It can be regarded as a $\mathcal{C}^\infty(M, \mathbb{C})$-linear map $\Omega^*(M, E) \to \Omega^*(M, \mathcal{W}_{\operatorname{cl}}^{1, 0} \otimes E)$. Extend it to a map
\begin{equation*}
	\Omega^*(M, \mathcal{W}_{\operatorname{cl}}^{1, 0} \otimes E) \to \Omega^*(M, \mathcal{W}_{\operatorname{cl}}^{1, 0} \otimes E)
\end{equation*}
by $\mathcal{C}^\infty(M, \mathcal{W}_{\operatorname{cl}}^{1, 0})$-linearity, which is still denoted by $I^E$ by abuse of notations.\par
Now, we quote Kapranov's theorems as follows (see also \cite{Yu2015}).

\begin{theorem}[Theorems 2.6 and 2.7.6 in \cite{Kap1999}]
	Let $M$ be a K\"ahler manifold.
	\begin{itemize}
		\item The following equality holds on $\Omega^{0, *}(M, \mathcal{W}_{\operatorname{cl}}^{1, 0})$:
		\begin{equation}
			\label{Equation 3.2}
			( \overline{\partial} + I )^2 = 0.
		\end{equation}
		\item Let $E$ be a Hermitian holomorphic vector bundle over $M$. Then on $\Omega^{0, *}(M, \mathcal{W}_{\operatorname{cl}}^{1, 0} \otimes E)$,
		\begin{equation}
			\label{Equation 3.3}
			( \overline{\partial} + I + I^E )^2 = 0.
		\end{equation}
	\end{itemize}
\end{theorem}

Indeed, $\overline{\partial} + I + I^E$ can be extended to a connection $D_{\operatorname{Kap}}^E$ in (\ref{Equation 3.4}). In this paper, we call $D_{\operatorname{Kap}}^E$ the \emph{Kapranov's connection} on $\mathcal{W}_{\operatorname{cl}}^{1, 0} \otimes E$.

\begin{proposition}[a generalization of Proposition 2.8 in \cite{ChaLeuLi2022b}]
	\label{Proposition 3.4}
	The connection
	\begin{equation}
		\label{Equation 3.4}
		D_{\operatorname{Kap}}^E := \nabla + \nabla^E - \delta^{1, 0} + I + I^E
	\end{equation}
	on $\mathcal{W}_{\operatorname{cl}}^{1, 0} \otimes E$ is flat, i.e. $(D_{\operatorname{Kap}}^E)^2 = 0$.
\end{proposition}
\begin{proof}
	Let $JE$ be the holomorphic jet bundle of $E$. Recall that there is a natural flat holomorphic connection $\nabla_{\operatorname{G}}^E$ on $JE$, known as the \emph{Grothendieck connection}. The connections $\nabla$ and $\nabla^E$ induce a smooth splitting $\eta_E: JE \cong \mathcal{W}_{\operatorname{cl}}^{1, 0} \otimes E$. Indeed, $D_{\operatorname{Kap}}^E$ is constructed by this splitting so that we have a cochain isomorphism:
	\begin{equation}
		\label{Equation 3.5}
		(\Omega^*(M, JE), \nabla_{\operatorname{G}}^E + \overline{\partial}) \cong (\Omega^*( M, \mathcal{W}_{\operatorname{cl}}^{1, 0} \otimes E ), D_{\operatorname{Kap}}^E).
	\end{equation}
\end{proof}

\begin{remark}
	Under the isomorphism (\ref{Equation 3.5}), $\nabla_{\operatorname{G}}^E$ is identified with $\nabla^{1, 0} - \delta^{1, 0}$ while the holomorphic structure $\overline{\partial}$ on $JE$ is identified with the operator $\overline{\partial} + I + I^E$ on $\Omega^*(M, \mathcal{W}_{\operatorname{cl}}^{1, 0} \otimes E)$. Note that the term $\overline{\partial}$ in the operator $\overline{\partial} + I + I^E$ denotes the holomorphic structure on $\mathcal{W}_{\operatorname{cl}}^{1, 0} \otimes E$ instead.
\end{remark}

Now, we prove (a reformulation of) the statement that there exists a quasi-isomorphism
\begin{equation*}
	(\Omega^{0, *}(M, E), \overline{\partial}) \hookrightarrow (\Omega^*(M, JE), \nabla_{\operatorname{G}}^E + \overline{\partial}).
\end{equation*}

\begin{proposition}
	There exists a canonical quasi-isomorphism
	\begin{equation*}
		(\Omega^{0, *}(M, E), \overline{\partial}) \hookrightarrow (\Omega^*(M, \mathcal{W}_{\operatorname{cl}}^{1, 0} \otimes E), D_{\operatorname{Kap}}^E).
	\end{equation*}
\end{proposition}
\begin{proof}
	Let $\Phi^E$ be the cochain isomorphism in Lemma \ref{Lemma 3.7} (to be proved below). By Lemma \ref{Lemma 3.1}, the restriction of $(\Phi^E)^{-1}$ onto $\Omega^{0, *}(M, E)$ gives the desired quasi-isomorphism.
\end{proof}

\begin{lemma}
	\label{Lemma 3.7}
	The map
	\begin{equation*}
		\Phi^E: (\Omega^*(M, \mathcal{W}_{\operatorname{cl}}^{1, 0} \otimes E), D_{\operatorname{Kap}}^E) \to (\Omega^*(M, \mathcal{W}_{\operatorname{cl}}^{1, 0} \otimes E), \overline{\partial} - \delta^{1, 0})
	\end{equation*}
	given by $s \mapsto s - (\delta^{1, 0})^{-1} (D_{\operatorname{Kap}}^E - (\overline{\partial} - \delta^{1, 0})) s$ is a cochain isomorphism.
\end{lemma}
\begin{proof}
	Write $D = D_{\operatorname{Kap}}^E$ and $\underline{D} = D_{\operatorname{Kap}}^E - (\overline{\partial} - \delta^{1, 0}) = \nabla^{1, 0} + I + I^E$. Then $\Phi^E(s) = s - (\delta^{1, 0})^{-1} \underline{D} s$.
	\begin{equation*}
		\Phi^E D - (\overline{\partial} - \delta^{1, 0}) \Phi^E = \underline{D} - (\delta^{1, 0})^{-1} \underline{D} D + (\overline{\partial} - \delta^{1, 0}) (\delta^{1, 0})^{-1} \underline{D}.
	\end{equation*}
	Since $D^2 = (\overline{\partial} - \delta^{1, 0})^2 = 0$, $-(\delta^{1, 0})^{-1} \underline{D} D = (\delta^{1, 0})^{-1} (\overline{\partial} - \delta^{1, 0}) D = (\delta^{1, 0})^{-1} (\overline{\partial} - \delta^{1, 0}) \underline{D}$. Then
	\begin{equation*}
		\Phi^E D - (\overline{\partial} - \delta^{1, 0}) \Phi^E = (1 + (\delta^{1, 0})^{-1} (\overline{\partial} - \delta^{1, 0}) + (\overline{\partial} - \delta^{1, 0}) (\delta^{1, 0})^{-1}) \underline{D} = \pi_{0, *} \underline{D} = 0
	\end{equation*}
	by Lemma \ref{Lemma 3.1}. Thus, $\Phi^E$ is a cochain map. It remains to show that $\Phi^E$ is a linear isomorphism.\par
	Indeed, the space $\Omega^*(M, \mathcal{W}_{\operatorname{cl}}^{1, 0} \otimes E)$ has a decreasing filtration whose $r$th filtered piece is $\Omega^*(M, \operatorname{Sym}^{\geq r} T^\vee \otimes E)$. The image of $\Omega^*(M, \operatorname{Sym}^{\geq r} T^\vee \otimes E)$ under $(\delta^{1, 0})^{-1} \underline{D}$ is contained in $\Omega^*(M, \operatorname{Sym}^{\geq r+1} T^\vee \otimes E)$. We can show that, for any $s \in \Omega^*(M, \mathcal{W}_{\operatorname{cl}}^{1, 0} \otimes E)$, there is a unique solution $P_s \in \Omega^*(M, \mathcal{W}_{\operatorname{cl}}^{1, 0} \otimes E)$ to the equation
	\begin{equation*}
		P_s - (\delta^{1, 0})^{-1} \underline{D} P_s = s.
	\end{equation*}
	This is done by solving the graded pieces of $P_s$ iteratively. Thus, the proof is complete.
\end{proof}

From now on, for $s \in \mathcal{C}^\infty(M, E)$, let $P_s = (\Phi^E)^{-1}(s) \in \mathcal{C}^\infty(M, \mathcal{W}_{\operatorname{cl}}^{1, 0} \otimes E)$. We can show that
\begin{equation*}
	P_s = \sum_{r=0}^\infty (\widetilde{\nabla}^{1, 0})^r s \quad \text{and} \quad \pi_{0, *}(P_s) = s.
\end{equation*}
It means that, under the identification via the smooth splitting $\eta_E: JE \cong \mathcal{W}_{\operatorname{cl}}^{1, 0} \otimes E$, $P_s$ is the jet prolongation of $s$ in the $T$-direction. Also, $s$ is $\overline{\partial}$-closed if and only if $P_s$ is $D_{\operatorname{Kap}}^E$-closed.\par
Considering the trivial Hermitian line bundle $M \times \mathbb{C}$ over $M$, we simply denote $D_{\operatorname{Kap}}^{M \times \mathbb{C}}$ by $D_{\operatorname{Kap}}$ and call it the \emph{Kapranov's connection} on $M$. In this case, the holomorphic jet bundle of $M \times \mathbb{C}$ is just the holomorphic jet bundle $J$ of $M$ and we denote $\nabla_{\operatorname{G}}^{M \times \mathbb{C}}$ by $\nabla_{\operatorname{G}}$.\par
Finally, note that the natural fibrewise action of $J$ on $JE$ is compatible with the connections $\nabla_G + \overline{\partial}$ and $\nabla_G^E + \overline{\partial}$. Via the smooth splittings $\eta_{M \times \mathbb{C}}: J \cong \mathcal{W}_{\operatorname{cl}}^{1, 0}$ and $\eta_E: JE \cong \mathcal{W}_{\operatorname{cl}}^{1, 0} \otimes E$, we can easily show that for any $p \in \mathbb{N}$, $a \in \Omega^p(M, \mathcal{W}_{\operatorname{cl}}^{1, 0})$ and $s \in \Omega^*(M, \mathcal{W}_{\operatorname{cl}}^{1, 0} \otimes E)$,
\begin{equation}
	\label{Equation 3.6}
	D_{\operatorname{Kap}}^E(as) = (D_{\operatorname{Kap}} a) s + (-1)^p a (D_{\operatorname{Kap}}^Es).
\end{equation}

\subsection{The Kapranov's connection $D_{\operatorname{Kap}}^{E_1, E_2}$ on $\operatorname{Hom}(\mathcal{W}_{\operatorname{cl}}^{1, 0} \otimes E_1, \mathcal{W}_{\operatorname{cl}}^{1, 0} \otimes E_2)$}
\label{Subsection 3.3}
\quad\par
Suppose that $E_1, E_2$ are Hermitian holomorphic vector bundles over $M$. Then the Kapranov's connections on $\mathcal{W}_{\operatorname{cl}}^{1, 0} \otimes E_1$ and $\mathcal{W}_{\operatorname{cl}}^{1, 0} \otimes E_2$ naturally induce a flat connection $D_{\operatorname{Kap}}^{E_1, E_2}$ on
\begin{equation*}
	\operatorname{Hom}( \mathcal{W}_{\operatorname{cl}}^{1, 0} \otimes E_1, \mathcal{W}_{\operatorname{cl}}^{1, 0} \otimes E_2 )
\end{equation*}
such that the smooth splittings $\eta_{E_i}: JE_i \to \mathcal{W}_{\operatorname{cl}}^{1, 0} \otimes E_i$ for $i = 1, 2$ yield a cochain isomorphism:
\begin{equation*}
	(\Omega^*(M, \operatorname{Hom}(JE_1, JE_2)), \nabla_{\operatorname{G}}^{E_1, E_2} + \overline{\partial}) \to (\Omega^*(M, \operatorname{Hom}( \mathcal{W}_{\operatorname{cl}}^{1, 0} \otimes E_1, \mathcal{W}_{\operatorname{cl}}^{1, 0} \otimes E_2 )), D_{\operatorname{Kap}}^{E_1, E_2}),
\end{equation*}
where $\nabla_{\operatorname{G}}^{E_1, E_2} := (\nabla_{\operatorname{G}}^{E_1})^\vee + \nabla_{\operatorname{G}}^{E_2}$. Explicitly, for $\phi \in \Omega^p(M, \operatorname{Hom}( \mathcal{W}_{\operatorname{cl}}^{1, 0} \otimes E_1, \mathcal{W}_{\operatorname{cl}}^{1, 0} \otimes E_2 ))$,
\begin{equation}
	\label{Equation 3.7}
	D_{\operatorname{Kap}}^{E_1, E_2} \phi = (\nabla + \nabla^{\operatorname{Hom}(E_1, E_2)})(\phi) - \delta^{1, 0}(\phi) + (I + I^{E_2}) \circ \phi - (-1)^p \phi \circ (I + I^{E_1}),
\end{equation}
where $\circ$ is the fibrewise composition and $\delta^{1, 0}(\phi) = \delta^{1, 0} \circ \phi - (-1)^p \phi \circ \delta^{1, 0}$. We call $D_{\operatorname{Kap}}^{E_1, E_2}$ the \emph{Kapranov's connection} on $\operatorname{Hom}( \mathcal{W}_{\operatorname{cl}}^{1, 0} \otimes E_1, \mathcal{W}_{\operatorname{cl}}^{1, 0} \otimes E_2 )$.\par
On the other hand, define a $\mathbb{C}$-linear operator $\widetilde{D}$ on $\Omega^*(M, \operatorname{Hom}( \mathcal{W}_{\operatorname{cl}}^{1, 0} \otimes E_1, \mathcal{W}_{\operatorname{cl}}^{1, 0} \otimes E_2 ))$ as follows. For $\phi \in \Omega^p(M, \operatorname{Hom}( \mathcal{W}_{\operatorname{cl}}^{1, 0} \otimes E_1, \mathcal{W}_{\operatorname{cl}}^{1, 0} \otimes E_2 ))$,
\begin{equation*}
	\widetilde{D}\phi := \overline{\partial} \phi - (-1)^p \phi \circ (I + I^{E_1}) = \overline{\partial} \circ \phi - (-1)^p \phi \circ (\overline{\partial} + I + I^{E_1}).
\end{equation*}
We can easily show by Proposition \ref{Proposition 3.4} that $\widetilde{D}^2 = 0$ on $\Omega^*(M, \operatorname{Hom}( \mathcal{W}_{\operatorname{cl}}^{1, 0} \otimes E_1, \mathcal{W}_{\operatorname{cl}}^{1, 0} \otimes E_2 ))$. Also, the smooth splitting $\eta_{E_1}: JE_1 \cong \mathcal{W}_{\operatorname{cl}}^{1, 0} \otimes E_1$ yields a cochain isomorphism:
\begin{equation*}
	(\Omega^{0, *}(M, \operatorname{Hom}(JE_1, E_2)), \overline{\partial}) \to (\Omega^{0, *}(M, \operatorname{Hom}( \mathcal{W}_{\operatorname{cl}}^{1, 0} \otimes E_1, E_2 )), \widetilde{D}).
\end{equation*}
Now, we prove (a reformulation of) the statement that there exists a quasi-isomorphism
\begin{equation*}
	(\Omega^{0, *}(M, \operatorname{Hom}(JE_1, E_2)), \overline{\partial}) \hookrightarrow (\Omega^*(M, \operatorname{Hom}(JE_1, JE_2)), \nabla_{\operatorname{G}}^{E_1, E_2} + \overline{\partial}).
\end{equation*}

\begin{proposition}
	\label{Proposition 3.8}
	There exists a canonical quasi-isomorphism
	\begin{equation*}
		(\Omega^{0, *}(M, \operatorname{Hom}( \mathcal{W}_{\operatorname{cl}}^{1, 0} \otimes E_1, E_2 )), \widetilde{D}) \hookrightarrow (\Omega^*(M, \operatorname{Hom}( \mathcal{W}_{\operatorname{cl}}^{1, 0} \otimes E_1, \mathcal{W}_{\operatorname{cl}}^{1, 0} \otimes E_2 )), D_{\operatorname{Kap}}^{E_1, E_2}).
	\end{equation*}
\end{proposition}
\begin{proof}
	Let $\Phi^{E_1, E_2}$ be the cochain isomorphism in Lemma \ref{Lemma 3.10}. By Lemma \ref{Lemma 3.9}, the restriction of $(\Phi^{E_1, E_2})^{-1}$ onto $\Omega^{0, *}(M, \operatorname{Hom}( \mathcal{W}_{\operatorname{cl}}^{1, 0} \otimes E_1, E_2 ))$ gives the desired quasi-isomorphism.
\end{proof}

Lemmas \ref{Lemma 3.9} and \ref{Lemma 3.10} will be proved below. We first note that, since $\overline{\partial} \circ \delta^{1, 0} + \delta^{1, 0} \circ \overline{\partial} = 0$ on $\Omega^*(M, \mathcal{W}_{\operatorname{cl}}^{1, 0} \otimes E_2)$, $(\widetilde{D} - \delta^{1, 0} \circ)^2 = 0$ on $\Omega^*(M, \operatorname{Hom}( \mathcal{W}_{\operatorname{cl}}^{1, 0} \otimes E_1, \mathcal{W}_{\operatorname{cl}}^{1, 0} \otimes E_2 ))$.

\begin{lemma}
	\label{Lemma 3.9}
	For any $\phi \in \Omega^*(M, \operatorname{Hom}( \mathcal{W}_{\operatorname{cl}}^{1, 0} \otimes E_1, \mathcal{W}_{\operatorname{cl}}^{1, 0} \otimes E_2 ))$, regarded as a map
	\begin{equation*}
		\Omega^*(M, \mathcal{W}_{\operatorname{cl}}^{1, 0} \otimes E_1) \to \Omega^*(M, \mathcal{W}_{\operatorname{cl}}^{1, 0} \otimes E_1),
	\end{equation*}
	the following equality holds:
	\begin{equation}
		\label{Equation 3.8}
		\phi - \pi_{0, *} \circ \phi = (\delta^{1, 0} \circ (\delta^{1, 0})^{-1} \circ \phi - \widetilde{D}((\delta^{1, 0})^{-1} \circ \phi)) + (\delta^{1, 0})^{-1} \circ (\delta^{1, 0} \circ \phi - \widetilde{D}\phi).
	\end{equation}
	In particular, the inclusion
	\begin{equation*}
		(\Omega^{0, *}(M, \operatorname{Hom}( \mathcal{W}_{\operatorname{cl}}^{1, 0} \otimes E_1, E_2 )), \widetilde{D}) \hookrightarrow (\Omega^*(M, \operatorname{Hom}( \mathcal{W}_{\operatorname{cl}}^{1, 0} \otimes E_1, \mathcal{W}_{\operatorname{cl}}^{1, 0} \otimes E_2 )), \widetilde{D} - \delta^{1, 0} \circ)
	\end{equation*}
	is a quasi-isomorphism.
\end{lemma}
\begin{proof}
	Let $p \in \mathbb{N}$ and $\phi \in \Omega^p(M, \operatorname{Hom}( \mathcal{W}_{\operatorname{cl}}^{1, 0} \otimes E_1, \mathcal{W}_{\operatorname{cl}}^{1, 0} \otimes E_2 ))$. Then the right hand side of (\ref{Equation 3.8}) is the sum of the following two terms:
	\begin{itemize}
		\item $(\delta^{1, 0} - \overline{\partial}) \circ (\delta^{1, 0})^{-1} \circ \phi + (\delta^{1, 0})^{-1} \circ (\delta^{1, 0} - \overline{\partial}) \circ \phi$, and
		\item $(-1)^{p - 1} ((\delta^{1, 0})^{-1} \circ \phi) \circ (\overline{\partial} - (I + I^{E_1})) + (-1)^p (\delta^{1, 0})^{-1} \circ  (\phi \circ (\overline{\partial} - (I + I^{E_1})))$.
	\end{itemize}
	The first term is equal to $\phi - \pi_{0, *} \circ \phi$ by Lemma \ref{Lemma 3.1} and the second term vanishes obviously.\par
	It is clear that the inclusion $\Omega^{0, *}(M, \operatorname{Hom}( \mathcal{W}_{\operatorname{cl}}^{1, 0} \otimes E_1, E_2 )) \hookrightarrow \Omega^*(M, \operatorname{Hom}( \mathcal{W}_{\operatorname{cl}}^{1, 0} \otimes E_1, \mathcal{W}_{\operatorname{cl}}^{1, 0} \otimes E_2 ))$ and $\pi_{0, *} \circ: \Omega^*(M, \operatorname{Hom}( \mathcal{W}_{\operatorname{cl}}^{1, 0} \otimes E_1, \mathcal{W}_{\operatorname{cl}}^{1, 0} \otimes E_2 )) \to \Omega^{0, *}(M, \operatorname{Hom}( \mathcal{W}_{\operatorname{cl}}^{1, 0} \otimes E_1, E_2 ))$ are cochain maps. If $\phi \in \Omega^{0, *}(M, \operatorname{Hom}( \mathcal{W}_{\operatorname{cl}}^{1, 0} \otimes E_1, E_2 ))$, then $\pi_{0, *} \circ \phi = \phi$. Then by (\ref{Equation 3.8}), the inclusion is a quasi-isomorphism.
\end{proof}

\begin{lemma}
	\label{Lemma 3.10}
	The map
	\begin{align*}
		\Phi^{E_1, E_2}: & (\Omega^*(M, \operatorname{Hom}( \mathcal{W}_{\operatorname{cl}}^{1, 0} \otimes E_1, \mathcal{W}_{\operatorname{cl}}^{1, 0} \otimes E_2 )), D_{\operatorname{Kap}}^{E_1, E_2})\\
		\to & (\Omega^*(M, \operatorname{Hom}( \mathcal{W}_{\operatorname{cl}}^{1, 0} \otimes E_1, \mathcal{W}_{\operatorname{cl}}^{1, 0} \otimes E_2 )), \widetilde{D} - \delta^{1, 0} \circ)
	\end{align*}
	given by $\phi \mapsto \phi - (\delta^{1, 0})^{-1} \circ ((D_{\operatorname{Kap}}^{E_1, E_2}\phi - (\widetilde{D}\phi - \delta^{1, 0} \circ \phi))$ is a cochain isomorphism.
\end{lemma}
\begin{proof}
	The idea of proof is the same as that of Lemma \ref{Lemma 3.7}. For the ease of notations, we write $D = D_{\operatorname{Kap}}^{E_1, E_2}$ and $\underline{D} = D_{\operatorname{Kap}}^{E_1, E_2} - (\widetilde{D} - \delta^{1, 0} \circ)$. Then $\Phi^{E_1, E_2}(\phi) = \phi - (\delta^{1, 0})^{-1} \circ (\underline{D}\phi)$, and if $p \in \mathbb{N}$ and $\phi \in \Omega^p(M, \operatorname{Hom}( \mathcal{W}_{\operatorname{cl}}^{1, 0} \otimes E_1, \mathcal{W}_{\operatorname{cl}}^{1, 0} \otimes E_2 ))$, then
	\begin{equation*}
		\psi := \underline{D} \phi = \nabla^{1, 0} \phi + (-1)^p \phi \circ \delta^{1, 0} + (I + I^{E_2}) \circ \phi,
	\end{equation*}
	and thus $\pi_{0, *} \circ \psi = 0$. Together with the fact that $D^2 = (\widetilde{D} - \delta^{1, 0} \circ)^2 = 0$ and (\ref{Equation 3.8}), we have
	\begin{align*}
		& \Phi^{E_1, E_2} (D\phi) - \widetilde{D} (\Phi^{E_1, E_2}(\phi)) + \delta^{1, 0} \circ \Phi^{E_1, E_2}(\phi)\\
		= & \psi - (\delta^{1, 0})^{-1} \circ (\underline{D} D\phi) + \widetilde{D} ((\delta^{1, 0})^{-1} \circ \psi) - \delta^{1, 0} \circ (\delta^{1, 0})^{-1} \circ \psi\\
		= & \psi + (\delta^{1, 0})^{-1} \circ (\widetilde{D}\psi - \delta^{1, 0} \circ \psi) + \widetilde{D}((\delta^{1, 0})^{-1} \circ \psi) - \delta^{1, 0} \circ (\delta^{1, 0})^{-1} \circ \psi\\
		= & \pi_{0, *} \circ \psi = 0.
	\end{align*}
	Thus, $\Phi^{E_1, E_2}$ is a cochain map.\par
	To show that $\Phi^{E_1, E_2}$ is a linear isomorphism, equip $\Omega^*(M, \operatorname{Hom}( \mathcal{W}_{\operatorname{cl}}^{1, 0} \otimes E_1, \mathcal{W}_{\operatorname{cl}}^{1, 0} \otimes E_2 ))$ with the decreasing filtration whose $r$th filtered piece is $\Omega^*(M, \operatorname{Hom}( \mathcal{W}_{\operatorname{cl}}^{1, 0} \otimes E_1, \operatorname{Sym}^{\geq r} T^\vee \otimes E_2 ))$. If $\phi$ is in the $r$th filtered piece, then $(\delta^{1, 0})^{-1} \circ (\underline{D}\phi)$ is in the $(r+1)$th filtered piece. By the iterative method again, we can show that for any $\phi \in \Omega^*(M, \operatorname{Hom}( \mathcal{W}_{\operatorname{cl}}^{1, 0} \otimes E_1, \mathcal{W}_{\operatorname{cl}}^{1, 0} \otimes E_2 ))$, there is a unique solution $P_\phi \in \Omega^*(M, \operatorname{Hom}( \mathcal{W}_{\operatorname{cl}}^{1, 0} \otimes E_1, \mathcal{W}_{\operatorname{cl}}^{1, 0} \otimes E_2 ))$ to the equation
	\begin{equation*}
		P_\phi - (\delta^{1, 0})^{-1} \circ (\underline{D} P_\phi) = \phi.
	\end{equation*}
\end{proof}

From now on, for $\phi \in \mathcal{C}^\infty(M, \operatorname{Hom}(\mathcal{W}_{\operatorname{cl}}^{1, 0} \otimes E_1, E_2))$, let $P_\phi := (\Phi^{E_1, E_2})^{-1}(\phi)$. Note that $\phi$ induces an operator
\begin{equation}
	\label{Equation 3.9}
	\widehat{\phi}: \mathcal{C}^\infty(M, E_1) \to \mathcal{C}^\infty(M, E_2), \quad s \mapsto \phi(P_s).
\end{equation}
We can show that
\begin{equation}
	\label{Equation 3.10}
	P_\phi(P_s) = P_{\widehat{\phi}(s)}
\end{equation}
for any $s \in \mathcal{C}^\infty(M, E)$, using the fact that $(\nabla^{1, 0} - \delta^{1, 0}) P_s = 0$ by Lemma \ref{Lemma 3.2}. It means that, under the identification via the smooth splittings $\eta_{E_i}: JE_i \cong \mathcal{W}_{\operatorname{cl}}^{1, 0} \otimes E_i$ for $i = 1, 2$, $P_\phi$ is the jet prolongation of $\phi$ in the $T$-direction. Also, $\phi$ is $\widetilde{D}$-closed if and only if $P_\phi$ is $D_{\operatorname{Kap}}^{E_1, E_2}$-closed.

\subsection{The bundle of holomorphic differential operators}
\label{Subsection 3.4}
\quad\par
Recall that $JE_1 = \varprojlim J^rE_1$ is the inverse limit of the bundles $J^rE_1$ of $r$th order holomorphic jets of $E_1$ for $r \in \mathbb{N}$. It induces an increasing filtration on $\operatorname{Hom}(JE_1, E_2)$:
\begin{center}
	\begin{tikzcd}
		\operatorname{Hom}(E_1, E_2) \subset \operatorname{Hom}(J^1E_1, E_2) \subset \cdots \operatorname{Hom}(J^rE_1, E_2) \subset \cdots \subset \operatorname{Hom}(JE_1, E_2).
	\end{tikzcd}
\end{center}
The bundle $D(E_1, E_2)$ of holomorphic differential operators from $E_1$ to $E_2$ is the direct limit of the holomorphic vector subbundles $D^r(E_1, E_2) := \operatorname{Hom}(J^rE_1, E_2)$ for $r \in \mathbb{N}$. It can be equivalently defined as the holomorphic vector subbundle
\begin{equation*}
	D(E_1, E_2) := \bigcup_{r \in \mathbb{N}} D^r(E_1, E_2)
\end{equation*}
of $\operatorname{Hom}(JE_1, E_2)$. Via the smooth splitting $\eta_{E_1}: JE_1 \to \mathcal{W}_{\operatorname{cl}}^{1, 0} \otimes E_1$, $D^r(E_1, E_2)$ is identified with a subbundle $\operatorname{Hom}^{\leq r}(\mathcal{W}_{\operatorname{cl}}^{1, 0} \otimes E_1, E_2)$ of $\operatorname{Hom}(\mathcal{W}_{\operatorname{cl}}^{1, 0} \otimes E_1, E_2)$ whose smooth sections annihilate smooth sections of $\operatorname{Sym}^{\geq r+1} T^\vee \otimes E_1$, and $D(E_1, E_2)$ is identified with
\begin{equation*}
	\operatorname{Hom}^{< \infty}(\mathcal{W}_{\operatorname{cl}}^{1, 0} \otimes E_1, E_2) := \bigcup_{r \in \mathbb{N}} \operatorname{Hom}^{\leq r}(\mathcal{W}_{\operatorname{cl}}^{1, 0} \otimes E_1, E_2).
\end{equation*}
We denote by $\mathcal{D}(\mathcal{E}_1, \mathcal{E}_2)$ the sheaf of holomorphic sections of $D(E_1, E_2)$. Note that for every $\phi \in \mathcal{C}^\infty(M, \operatorname{Hom}^{< \infty}(\mathcal{W}_{\operatorname{cl}}^{1, 0} \otimes E_1, E_2))$, which can be identified as a smooth section of $D(E_1, E_2)$, the operator $\widehat{\phi}$ in (\ref{Equation 3.9}) is a smooth differential operator. The assignment $\phi \mapsto \widehat{\phi}$ induces a bijection between holomorphic sections of $D(E_1, E_2)$ and holomorphic differential operators from $E_1$ to $E_2$. Therefore, we can identify them as the same objects and call $\mathcal{D}(\mathcal{E}_1, \mathcal{E}_2)$ the \emph{sheaf of holomorphic differential operators from $E_1$ to $E_2$}.\par
On the other hand, the canonical fibrewise actions of $J$ on $JE_i$'s induce a $\mathcal{C}^\infty(M, \mathbb{C})$-linear map
\begin{equation*}
	\mathcal{C}^\infty(M, J) \times \mathcal{C}^\infty(M, \operatorname{Hom}(JE_1, JE_2)) \to \mathcal{C}^\infty(M, \operatorname{Hom}(JE_1, JE_2)), \quad (a, \phi) \mapsto \operatorname{ad}(a)(\phi),
\end{equation*}
where $\operatorname{ad}(a)(\phi)(s) = a\phi(s) - \phi(as)$ for any $s \in \mathcal{C}^\infty(M, JE_1)$. For each $r \in \mathbb{N}$, we can obtain a vector subbundle $\operatorname{Hom}^{\leq r}(JE_1, JE_2)$ of $\operatorname{Hom}(JE_1, JE_2)$ consisting of elements $\phi \in \operatorname{Hom}((JE_1)_x, (JE_2)_x))$, where $x \in M$, for which $(\operatorname{ad}(a_0) \circ \cdots \operatorname{ad}(a_r))(\phi) = 0$ for any $a_0, ..., a_r \in J_x$. The direct limit
\begin{equation*}
	\operatorname{Hom}^{<\infty}(JE_1, JE_2) := \bigcup_{r \in \mathbb{N}} \operatorname{Hom}^{\leq r}(JE_1, JE_2)
\end{equation*}
is also a vector subbundle of $\operatorname{Hom}(JE_1, JE_2)$.\par
Analogously, we can define $\operatorname{ad}(a)(\phi)(s) = a\phi(s) - \phi(as)$ for any sections $a \in \mathcal{C}^\infty(M, \mathcal{W}_{\operatorname{cl}}^{1, 0})$, $\phi \in \mathcal{C}^\infty(M, \operatorname{Hom}(\mathcal{W}_{\operatorname{cl}}^{1, 0} \otimes E_1, \mathcal{W}_{\operatorname{cl}}^{1, 0} \otimes E_2))$ and $s \in \mathcal{C}^\infty(M, \mathcal{W}_{\operatorname{cl}}^{1, 0} \otimes E_1)$, and obtain subbundles $\operatorname{Hom}^{\leq r}(\mathcal{W}_{\operatorname{cl}}^{1, 0} \otimes E_1, \mathcal{W}_{\operatorname{cl}}^{1, 0} \otimes E_2)$ and $\operatorname{Hom}^{<\infty}(\mathcal{W}_{\operatorname{cl}}^{1, 0} \otimes E_1, \mathcal{W}_{\operatorname{cl}}^{1, 0} \otimes E_2)$ of $\operatorname{Hom}(\mathcal{W}_{\operatorname{cl}}^{1, 0} \otimes E_1, \mathcal{W}_{\operatorname{cl}}^{1, 0} \otimes E_2)$. By (\ref{Equation 3.6}), we see that for any $a \in \mathcal{C}^\infty(M, \mathcal{W}_{\operatorname{cl}}^{1, 0})$ and $\phi \in \mathcal{C}^\infty(M, \operatorname{Hom}(\mathcal{W}_{\operatorname{cl}}^{1, 0} \otimes E_1, \mathcal{W}_{\operatorname{cl}}^{1, 0} \otimes E_2))$,
\begin{equation}
	D_{\operatorname{Kap}}^{E_1, E_2}(\operatorname{ad}(a)(\phi)) = \operatorname{ad}(D_{\operatorname{Kap}}(a))(\phi) + \operatorname{ad}(a)(D_{\operatorname{Kap}}^{E_1, E_2}\phi).
\end{equation}
We can then deduce that, for each $r \in \mathbb{N}$, $\Omega^*(M, \operatorname{Hom}^{\leq r}( \mathcal{W}_{\operatorname{cl}}^{1, 0} \otimes E_1, \mathcal{W}_{\operatorname{cl}}^{1, 0} \otimes E_2 ))$ forms a subcomplex of $(\Omega^*(M, \operatorname{Hom}^{\leq r}( \mathcal{W}_{\operatorname{cl}}^{1, 0} \otimes E_1, \mathcal{W}_{\operatorname{cl}}^{1, 0} \otimes E_2 )), D_{\operatorname{Kap}}^{E_1, E_2})$. Hence, $\operatorname{Hom}^{< \infty}( \mathcal{W}_{\operatorname{cl}}^{1, 0} \otimes E_1, \mathcal{W}_{\operatorname{cl}}^{1, 0} \otimes E_2 )$ is also a subcomplex.

\begin{proposition}
	\label{Proposition 3.11}
	For each $r \in \mathbb{N}$, the map
	\begin{equation*}
		\mathcal{C}^\infty(M, \operatorname{Hom}(\mathcal{W}_{\operatorname{cl}}^{1, 0} \otimes E_1, E_2)) \to \mathcal{C}^\infty(M, \operatorname{Hom}(\mathcal{W}_{\operatorname{cl}}^{1, 0} \otimes E_1, \mathcal{W}_{\operatorname{cl}}^{1, 0} \otimes E_2)), \quad \phi \mapsto P_\phi,
	\end{equation*}
	restricts to a bijection between $\widetilde{D}$-closed sections of $\operatorname{Hom}^{\leq r}(\mathcal{W}_{\operatorname{cl}}^{1, 0} \otimes E_1, E_2)$ and $D_{\operatorname{Kap}}^{E_1, E_2}$-flat sections of $\operatorname{Hom}^{\leq r}( \mathcal{W}_{\operatorname{cl}}^{1, 0} \otimes E_1, \mathcal{W}_{\operatorname{cl}}^{1, 0} \otimes E_2 )$.
\end{proposition}
\begin{proof}
	By Proposition \ref{Proposition 3.8}, the map $\phi \mapsto P_\phi$ restricts to a bijection between $\widetilde{D}$-closed sections of $\operatorname{Hom}(\mathcal{W}_{\operatorname{cl}}^{1, 0} \otimes E_1, E_2)$ and $D_{\operatorname{Kap}}^{E_1, E_2}$-flat sections of $\operatorname{Hom}( \mathcal{W}_{\operatorname{cl}}^{1, 0} \otimes E_1, \mathcal{W}_{\operatorname{cl}}^{1, 0} \otimes E_2 )$.\par
	Let $\phi \in \mathcal{C}^\infty(M, \operatorname{Hom}(\mathcal{W}_{\operatorname{cl}}^{1, 0} \otimes E_1, E_2))$ be $\widetilde{D}$-closed. For any $f \in \mathcal{C}^\infty(M, \mathbb{C})$,  and $s \in \mathcal{C}^\infty(M, E_1)$,
	\begin{equation*}
		\operatorname{ad}(P_f)(P_\phi)(P_s) = P_{\operatorname{ad}(f)(\widehat{\phi})(s)},
	\end{equation*}
	where $\operatorname{ad}(f)(\widehat{\phi})(s) := f\widehat{\phi}(s) - \widehat{\phi}(fs)$, due to (\ref{Equation 3.10}) and the fact that $P_{fs} = P_fP_s$. It implies that $\phi \in \mathcal{C}^\infty(M, \operatorname{Hom}^{\leq r}(\mathcal{W}_{\operatorname{cl}}^{1, 0} \otimes E_1, E_2))$ if and only if $P_\phi \in \mathcal{C}^\infty(M, \operatorname{Hom}^{\leq r}( \mathcal{W}_{\operatorname{cl}}^{1, 0} \otimes E_1, \mathcal{W}_{\operatorname{cl}}^{1, 0} \otimes E_2 ))$. We are done.
\end{proof}

\section{Quantization of the classical category of Hermitian holomorphic vector bundles}
\label{Section 4}
We will introduce the Weyl bundle on the K\"ahler manifold $(M, \omega)$ in Subsection \ref{Subsection 4.1}, Fedosov's connections on $(M, \omega)$ in Subsection \ref{Subsection 4.2}, and formal quantizability of morphisms in $\mathsf{DQ}$ in Subsection \ref{Subsection 4.3}. They serve as preparations for the proof of Theorem \ref{Theorem 1.1} in Subsection \ref{Subsection 4.4}. In Subsection \ref{Subsection 4.5}, we will define non-formal quantizable morphisms, which will be further studied in Section \ref{Section 5}.

\subsection{The Weyl bundle on a K\"ahler manifold}
\label{Subsection 4.1}
\quad\par
We first introduce some basic notations in K\"ahler geometry. In local complex coordinates $(z^1, ..., z^n)$, we write $\omega = \omega_{\alpha\overline{\beta}} dz^\alpha \wedge d\overline{z}^\beta$ and let $(\omega^{\overline{\alpha}\beta})$ be the inverse of $(\omega_{\alpha\overline{\beta}})$. The curvature of the Levi-Civita connection $\nabla$ is locally written as
\begin{equation*}
	\nabla^2 \left( \frac{\partial}{\partial z^\mu} \right) = R_{\alpha\overline{\beta}\mu}^\nu dz^\alpha \wedge d\overline{z}^\beta \otimes \frac{\partial}{\partial z^\nu} \quad \text{and} \quad \nabla^2 \left( \frac{\partial}{\partial \overline{z}^\mu} \right) = R_{\alpha\overline{\beta}\overline{\mu}}^{\overline{\nu}} dz^\alpha \wedge d\overline{z}^\beta \otimes \frac{\partial}{\partial \overline{z}^\nu}.
\end{equation*}
The \emph{Weyl bundle} of $(M, \omega)$ is the infinite rank vector bundle $\mathcal{W} = \widehat{\operatorname{Sym}} T_\mathbb{C}^\vee[[\hbar]]$, where $\widehat{\operatorname{Sym}} T_\mathbb{C}^\vee$ is the completed symmetric algebra bundle of the complexified cotangent bundle $T_\mathbb{C}^\vee$. A smooth section of $\mathcal{W}$ is locally given by a formal power series
\begin{equation*}
	\sum_{r, l \geq 0} \sum_{i_1, ..., i_l \geq 0} \hbar^r a_{r, i_1, ..., i_l} y^{i_1} \cdots y^{i_l},
\end{equation*}
where $a_{r, i_1, ..., i_l}$ are local smooth complex valued functions, $(x^1, ..., x^{2n})$ are local real coordinates, $y^i$ denotes the covector $dx^i$ regarded as a section of $\mathcal{W}$ and we suppress the notations of symmetric products in the above expression. There are three $\mathcal{C}^\infty(M, \mathbb{C})[[\hbar]]$-linear operators $\delta, \delta^{-1}, \pi_0$ on $\Omega^*(M, \mathcal{W})$ defined as follows: for a local section $a = y^{i_1} \cdots y^{i_l} dx^{j_1} \wedge \cdots \wedge dx^{j_m}$,
\begin{align*}
	\delta a = dx^k \wedge \frac{\partial a}{\partial y^k}, \quad \delta^{-1} a =
	\begin{cases}
		\frac{1}{l+m} y^k \iota_{\partial_{x^k}} a & \text{ if } l + m > 0;\\
		0 & \text{ if } l + m = 0,
	\end{cases}
	\quad \pi_0 (a) =
	\begin{cases}
		0  & \text{ if } l + m > 0;\\
		a & \text{ if } l + m = 0,
	\end{cases}
\end{align*}
The equality
\begin{equation}
	\label{Equation 4.1}
	\operatorname{Id} - \pi_0 = \delta \circ \delta^{-1} + \delta^{-1} \circ \delta
\end{equation}
holds on $\Omega^*(M, \mathcal{W})$.\par
The complex structure on $M$ gives rise to a subbundle $\mathcal{W}^{1, 0} = \widehat{\operatorname{Sym}} T^\vee [[\hbar]]$ of $\mathcal{W}$ and three $\mathcal{C}^\infty(M, \mathbb{C})[[\hbar]]$-linear operators $\delta^{1, 0}, (\delta^{1, 0})^{-1}, \pi_{0, *}$ on $\Omega^*(M, \mathcal{W})$ defined as follows: for a local section $a = w^{\mu_1} \cdots w^{\mu_l} \overline{w}^{\nu_1} \cdots \overline{w}^{\nu_m} dz^{\alpha_1} \wedge \cdots \wedge dz^{\alpha_p} \wedge d\overline{z}^{\beta_1} \wedge \cdots \wedge d\overline{z}^{\beta_q}$,
\begin{align*}
	\delta^{1, 0} a = dz^\mu \wedge \frac{\partial a}{\partial w^\mu},\quad
	(\delta^{1, 0})^{-1} a =
	\begin{cases}
		\frac{1}{l+p} w^\mu \iota_{\partial_{z^\mu}} a & \text{ if } l + p > 0;\\
		0 & \text{ if } l + p = 0,
	\end{cases}
	\quad \pi_{0, *}(a) =
	\begin{cases}
		0 & \text{ if } l + p > 0;\\
		a & \text{ if } l + p = 0,
	\end{cases}
\end{align*}
Here, we denote by $w^\mu$ (resp. $\overline{w}^\mu$) the covector $dz^\mu$ (resp. $d\overline{z}^\mu$) regarded as a section of $\mathcal{W}$. The equality $\operatorname{Id} - \pi_{0, *} = \delta^{1, 0} \circ (\delta^{1, 0})^{-1} + (\delta^{1, 0})^{-1} \circ \delta^{1, 0}$ holds on $\Omega^*(M, \mathcal{W})$. We can define the antiholomorphic counterparts $\mathcal{W}^{0, 1}$, $\delta^{0, 1}$, $(\delta^{0, 1})^{-1}$, $\pi_{*, 0}$ of $\mathcal{W}^{1, 0}$, $\delta^{1, 0}$, $(\delta^{1, 0})^{-1}$, $\pi_{0, *}$ respectively.\par
We equip $\mathcal{W}$ with the \emph{fibrewise anti-Wick product} $\star$ defined as follows (c.f., for instance, \cite{BorWal1997, Neu2003}). For $a, b \in \mathcal{C}^\infty(M, \mathcal{W})$,
\begin{equation}
	\label{Equation 4.2}
	a \star b := \sum_{r=0}^\infty \frac{\hbar^r}{r!} \omega^{\overline{\nu}_1\mu_1} \cdots \omega^{\overline{\nu}_r\mu_r} \frac{\partial^r a}{\partial \overline{w}^{\nu_1} \cdots \partial \overline{w}^{\nu_r}} \frac{\partial^r b}{\partial w^{\mu_1} \cdots \partial w^{\mu_r}}.
\end{equation}
We can naturally extend $\star$ to a $\mathcal{C}^\infty(M, \mathbb{C})[[\hbar]]$-bilinear map $\Omega^*(M, \mathcal{W}) \times \Omega^*(M, \mathcal{W}) \to \Omega^*(M, \mathcal{W})$ such that for all $\alpha, \beta \in \Omega^*(M, \mathcal{W})$ and $a, b \in \mathcal{C}^\infty(M, \mathcal{W})$, $(\alpha \otimes a) \star (\beta \otimes b) = (\alpha \wedge \beta) \otimes (a \star b)$. For other operations on vector bundles (of possibly infinite rank) appeared later in this paper, we also extend them onto vector bundle valued forms in a similar way.\par
By abuse of notations, we denote by $\nabla$ the connection on $\mathcal{W}$ induced by the Levi-Civita connection. Then $(\Omega^*(M, \mathcal{W}), R, \nabla, \tfrac{1}{\hbar}[\quad, \quad]_\star)$ forms a curved dgla, where $[\quad, \quad]_\star$ is the graded commutator of $\star$ and $R \in \Omega^{1, 1}(M, T^\vee \otimes \overline{T}^\vee)$ is given by $R = -\omega_{\eta\overline{\nu}} R_{\alpha\overline{\beta}\mu}^\eta dz^\alpha \wedge d\overline{z}^\beta \otimes w^\mu \overline{w}^\nu$.

\subsection{Fedosov's connections via Kapranov's $L_\infty$-structures}
\label{Subsection 4.2}
\quad\par
It was pointed out by Chan-Leung-Li \cite{ChaLeuLi2022b} that Kapranov's $L_\infty$-algebra structure (which is encoded by $I$ as in Subsection \ref{Subsection 3.2}) can be extended to a flat connection involved in Fedosov's quantization of the K\"ahler manifold $(M, \omega)$. It turns out that Kapranov's $L_\infty$-module structures (which are encoded by $I^E$'s as in Subsection \ref{Subsection 3.2}) also have similar extensions, which are hidden in Neumaier-Waldmann's construction ((27) in \cite{NeuWal2003}).\par
To see this, we start with defining the $\mathcal{W}$-valued $(0, 1)$-form $\widetilde{I} = \sum_{r=2}^\infty \widetilde{I}_{(r)}$, where for $r \geq 2$,
\begin{equation*}
	\widetilde{I}_{(r)} := (\widetilde{\nabla}^{1, 0})^{r-2} (\delta^{1, 0})^{-1}(R) \in \Omega^{0, 1}(M, \operatorname{Sym}^r T^\vee \otimes \overline{T}^\vee).
\end{equation*}
By (a variant \footnote{Indeed, we have modified Chan-Leung-Li’s original argument by taking the anti-Wick ordering instead of Wick ordering. In later (sub)sections, their definition of (formal) quantizable functions, their construction of fibrewise Bargmann-Fock actions, etc., are similarly modified.} of) Theorem 2.17 in \cite{ChaLeuLi2022b},
\begin{equation}
	\gamma := (\delta^{0, 1})^{-1} \omega + (\delta^{1, 0})^{-1} \omega + \widetilde{I}
\end{equation}
is a solution in $\Omega^1(M, \mathcal{W})$ to the equation
\begin{equation}
	R + \nabla \gamma + \tfrac{1}{2\hbar} [\gamma, \gamma]_\star = -\omega.
\end{equation}

Consider three Hermitian holomorphic vector bundles $E_1, E_2, E_3$ on $M$. The usual composition of smooth sections of hom-bundles, the wedge product of forms on $M$ and the fibrewise anti-Wick product $\star$ naturally define a $\mathbb{C}[[\hbar]]$-bilinear map
\begin{equation}
	\label{Equation 4.5}
	\star: \Omega^*(M, \mathcal{W} \otimes \operatorname{Hom}(E_2, E_3)) \times \Omega^*(M, \mathcal{W} \otimes \operatorname{Hom}(E_1, E_2)) \to \Omega^*(M, \mathcal{W} \otimes \operatorname{Hom}(E_1, E_3)).
\end{equation}
We also define a $\mathbb{C}[[\hbar]]$-bilinear map
\begin{equation}
	[\quad, \quad]_\star: \Omega^*(M, \mathcal{W}) \times \Omega^*(M, \mathcal{W} \otimes \operatorname{Hom}(E_1, E_2)) \to \Omega^*(M, \mathcal{W} \otimes \operatorname{Hom}(E_1, E_2))
\end{equation}
as follows. If $\alpha \in \Omega^p(M, \mathcal{W})$ and $\phi \in \Omega^q(M, \mathcal{W} \otimes \operatorname{Hom}(E_1, E_2))$, then
\begin{equation*}
	[\alpha, \phi]_\star := (\alpha \otimes \operatorname{Id}_{E_2}) \star \phi - (-1)^{pq} \phi \star (\alpha \otimes \operatorname{Id}_{E_1}).
\end{equation*}
We will now show that there are flat connections on $\mathcal{W} \otimes \operatorname{Hom}(E_i, E_j)$'s compatible with $\star$. The following proposition is a generalization of Theorem 2 in \cite{NeuWal2003} in the case of anti-Wick ordering.

\begin{proposition}
	\label{Proposition 4.1}
	Suppose $E_1, E_2$ are Hermitian holomorphic vector bundles. Then the following connection on $\mathcal{W} \otimes \operatorname{Hom}(E_1, E_2)$ is flat:
	\begin{equation}
		D^{E_1, E_2} := \nabla + \nabla^{\operatorname{Hom}(E_1, E_2)} + \tfrac{1}{\hbar} [\gamma, \quad]_\star + I^{E_2} \star - \star I^{E_1}.
	\end{equation}
\end{proposition}

\begin{remark}
	Indeed, as $\delta^{1, 0} = -\tfrac{1}{\hbar}[(\delta^{0, 1})^{-1} \omega, \quad]_\star$ and $\delta^{0, 1} = -\tfrac{1}{\hbar}[(\delta^{1, 0})^{-1} \omega, \quad]_\star$, we have an alternative expression:
	\begin{equation*}
		D^{E_1, E_2} = \nabla + \nabla^{\operatorname{Hom}(E_1, E_2)} - \delta + \tfrac{1}{\hbar}[\widetilde{I}, \quad]_\star + I^{E_2} \star - \star I^{E_1}.
	\end{equation*}
\end{remark}

Before proving Proposition \ref{Proposition 4.1}, we need a lemma.

\begin{lemma}
	\label{Lemma 4.3}
	$(\nabla^{1, 0} - \delta^{1, 0}) (\widetilde{I}) = -R$ and if $E$ is a Hermitian holomorphic vector bundle over $M$, then $(\nabla^{1, 0} - \delta^{1, 0}) (I^E) = -R^E$.
\end{lemma}
\begin{proof}
	The first equality is known to hold in the proof of Theorem 2.17 in \cite{ChaLeuLi2022b}, which is similar to the proof of Lemma \ref{Lemma 3.2}. Now we prove the second equality. As $\delta^{1, 0}R^E = 0$ and $\pi_{0, *}(R^E) = 0$, $\delta^{1, 0} (I_{(1)}^E) = R^E - \pi_{0, *}(R^E) - (\delta^{1, 0})^{-1} \circ \delta^{1, 0} (R^E) = R^E$. Then $(\nabla^{1, 0} \circ \delta^{1, 0})(I_{(1)}^E) = \nabla^{1, 0}(R^E) = 0$ by Bianchi identity. The second equality holds by Lemma \ref{Lemma 3.2}.
\end{proof}

\begin{proof}[\myproof{Proposition}{\ref{Proposition 4.1}}]
	The $(1, 0)$- and $(0, 1)$-parts of $D := D^{E_1, E_2}$ are
	\begin{align*}
		D^{1, 0} = & \nabla^{1, 0} - \delta^{1, 0},\\
		D^{0, 1} = & \overline{\partial} - \delta^{0, 1} + \tfrac{1}{\hbar} [\widetilde{I}, \quad]_\star + I^{E_2} \star - \star I^{E_1}
	\end{align*}
	respectively. We consider different components of the curvature $D^2$ separately.
	\begin{enumerate}
		\item The $(2, 0)$-part of $D^2$, i.e. $(D^{1, 0})^2$, vanishes, following from the fact that $(\nabla^{1, 0} - \delta^{1, 0})^2 = 0$.
		\item The $(1, 1)$-part of $D^2$, i.e. $[D^{1, 0}, D^{0, 1}]$,  is the sum of the terms
		\begin{equation*}
			[D^{1, 0}, \overline{\partial} - \delta^{0, 1}] = \tfrac{1}{\hbar} [R, \quad]_\star + R^{E_2} \star - \star R^{E_1}
		\end{equation*}
		and
		\begin{align*}
			& [D^{1, 0}, \tfrac{1}{\hbar} [\widetilde{I}, \quad]_\star + I^{E_2} \star - \star I^{E_1}]\\
			= & \tfrac{1}{\hbar} [ (\nabla^{1, 0} - \delta^{1, 0}) (\widetilde{I}), \quad ]_\star + ((\nabla^{1, 0} - \delta^{1, 0}) I^{E_2}) \star - \star ((\nabla^{1, 0} - \delta^{1, 0}) I^{E_1}),
		\end{align*}
		and hence vanishes by Lemma \ref{Lemma 4.3}.
		\item The $(0, 2)$-part of $D^2$, i.e. $(D^{0, 1})^2$, is given by
		\begin{equation*}
			(\overline{\partial} - \delta^{0, 1})^2 + [\overline{\partial} - \delta^{0, 1}, \tfrac{1}{\hbar} [\widetilde{I}, \quad]_\star + I^{E_2} \star - \star I^{E_1}] + (\tfrac{1}{\hbar} [\widetilde{I}, \quad]_\star + I^{E_2} \star - \star I^{E_1})^2.
		\end{equation*}
		First, we know that $(\overline{\partial} - \delta^{0, 1})^2 = 0$. Second, we have $\delta^{0, 1}\widetilde{I} = 0$ by Lemma 2.16 in \cite{ChaLeuLi2022b} and clearly $\delta^{0, 1}I^{E_i} = 0$ for $i = 1, 2$, whence
		\begin{equation*}
			[\overline{\partial} - \delta^{0, 1}, \tfrac{1}{\hbar} [\widetilde{I}, \quad]_\star + I^{E_2} \star - \star I^{E_1}] = \tfrac{1}{\hbar} [ \overline{\partial} \widetilde{I}, \quad ]_\star + (\overline{\partial} I^{E_1}) \star - \star (\overline{\partial} I^{E_1}).
		\end{equation*}
		Third, we have
		\begin{align*}
			& (\tfrac{1}{\hbar} [\widetilde{I}, \quad]_\star + I^{E_2} \star - \star I^{E_1})^2\\
			= & \tfrac{1}{\hbar} [ \tfrac{1}{\hbar} [\widetilde{I}, \widetilde{I}]_\star, \quad ]_\star + ( [I, I^{E_2}] + \tfrac{1}{2}[I^{E_2}, I^{E_2}] ) \star - \star (  [I, I^{E_1}] + \tfrac{1}{2}[I^{E_1}, I^{E_1}] ).
		\end{align*}
		It is known in the proof of Theorem 2.17 in \cite{ChaLeuLi2022b} that $\frac{1}{\hbar} [ \overline{\partial} \widetilde{I} + \frac{1}{\hbar} [\widetilde{I}, \widetilde{I}]_\star, \quad ]_\star = 0$. For $i = 1, 2$, the term $\overline{\partial}(I^{E_i}) + [I, I^{E_i}] + \tfrac{1}{2}[I^{E_i}, I^{E_i}]$ is indeed the subtraction of (\ref{Equation 3.2}) from (\ref{Equation 3.3}) and is hence equal to zero. Therefore, $(D^{0, 1})^2 = 0$.
	\end{enumerate}
\end{proof}

\begin{proposition}
	There exists a canonical quasi-isomorphism:
	\begin{equation*}
		(\mathcal{C}^\infty(M, \operatorname{Hom}(E_1, E_2))[[\hbar]], 0) \to (\Omega^*(M, \mathcal{W} \otimes \operatorname{Hom}(E_1, E_2)), D^{E_1, E_2}).
	\end{equation*}
\end{proposition}
\begin{proof}
	The inclusion
	\begin{equation*}
		(\mathcal{C}^\infty(M, \operatorname{Hom}(E_1, E_2))[[\hbar]], 0) \hookrightarrow (\Omega^*(M, \mathcal{W} \otimes \operatorname{Hom}(E_1, E_2)), -\delta)
	\end{equation*}
	is a quasi-isomorphism by (\ref{Equation 4.1}). Let $\Psi$ be the cochain isomorphism in Lemma \ref{Lemma 4.5}. Then the restriction of $\Psi^{-1}$ onto $\mathcal{C}^\infty(M, \operatorname{Hom}(E_1, E_2))[[\hbar]]$ gives the desired quasi-isomorphism.
\end{proof}

\begin{lemma}
	\label{Lemma 4.5}
	The map
	\begin{equation*}
		\Psi: (\Omega^*(M, \mathcal{W} \otimes \operatorname{Hom}(E_1, E_2)), D^{E_1, E_2}) \to (\Omega^*(M, \mathcal{W} \otimes \operatorname{Hom}(E_1, E_2)), -\delta)
	\end{equation*}
	given by $\phi \mapsto \phi - \delta^{-1} (D^{E_1, E_2} + \delta) \phi$ is a cochain isomorphism.
\end{lemma}
\begin{proof}
	Write $D = D^{E_1, E_2}$ and $\underline{D} = D + \delta$. Then $\Psi(\phi) = \phi - \delta^{-1} \underline{D} \phi$. Since $D^2 = \delta^2 = 0$,
	\begin{equation*}
		\Psi D + \delta \Psi = D + \delta - \delta^{-1} \underline{D} D - \delta \delta^{-1} \underline{D} = \underline{D} - \delta^{-1} \delta \underline{D} - \delta \delta^{-1} \underline{D} = \pi_0 \underline{D} = 0
	\end{equation*}
	by (\ref{Equation 4.1}). Thus, $\Psi$ is a cochain map. It remains to show that $\Psi$ is a linear isomorphism.\par
	Indeed, we can assign a weight to each element in $\Omega^*(M, \mathcal{W} \otimes \operatorname{Hom}(E_1, E_2))$ by declaring that differential forms on $M$ and smooth sections of $\operatorname{Hom}(E_1, E_2)$ are of weight $0$, $\omega^\alpha, \overline{w}^\beta$ are of weight $1$ and $\hbar$ is of weight $2$. The space is then equipped with a decreasing filtration whose $r$th filtered piece contains elements of weight at least $r$ \footnote{The weight assignment and the filtration in this proof are not the same as those in Subsection \ref{Subsection 4.3} -- we will no longer use them for the rest of this paper.}. The operator $\delta^{-1}\underline{D}$ increases weight by at least $1$. We can show that, for any $\phi \in \Omega^*(M, \mathcal{W} \otimes \operatorname{Hom}(E_1, E_2))$, there is a unique solution $O_\phi \in \Omega^*(M, \mathcal{W} \otimes \operatorname{Hom}(E_1, E_2))$ to the equation
	\begin{equation*}
		O_\phi - \delta^{-1}\underline{D} O_\phi = \phi.
	\end{equation*}
	This is done by solving the graded pieces of $O_\phi$ iteratively. Thus, the proof is complete.
\end{proof}

Elements $\phi \in \mathcal{C}^\infty(M, \operatorname{Hom}(E_1, E_2))[[\hbar]]$ are then in bijection with $D^{E_1, E_2}$-flat sections $O_\phi$ of $\mathcal{W} \otimes \operatorname{Hom}(E_1, E_2)$, where $O_\phi$ is uniquely determined by the condition that $\pi_0(O_\phi) = \phi$.

\begin{remark}
	\label{Remark 4.6}
	Using the same type of arguments as the proof of Lemma 2.5 in \cite{ChaLeuLi2023}, one can show that a $D^{E_1, E_2}$-flat section $O$ of $\mathcal{W} \otimes \operatorname{Hom}(E_1, E_2)$ must be of the form
	\begin{equation}
		O = \sum_{r=0}^\infty (\widetilde{\nabla}^{1, 0})^r \widetilde{O},
	\end{equation}
	where $\widetilde{O}$ is the $\mathcal{W}^{0, 1} \otimes \operatorname{Hom}(E_1, E_2)$-component of $O$.
\end{remark}

\begin{proposition}
	\label{Proposition 4.7}
	Suppose $E_1, E_2, E_3$ are Hermitian holomorphic vector bundles. Then for all $\phi \in \mathcal{C}^\infty(M, \mathcal{W} \otimes \operatorname{Hom}(E_1, E_2))$ and $\psi \in \mathcal{C}^\infty(M, \mathcal{W} \otimes \operatorname{Hom}(E_2, E_3))$,
	\begin{equation*}
		D^{E_1, E_3} (\psi \star \phi) = (D^{E_2, E_3} \psi) \star \phi + \psi \star (D^{E_1, E_2} \phi).
	\end{equation*}
\end{proposition}

Thus, $\star$ descends to a $\mathbb{C}[[\hbar]]$-linear map
\begin{equation}
	\label{Equation 4.9}
	\mathcal{C}^\infty(M, \operatorname{Hom}(E_2, E_3))[[\hbar]] \times \mathcal{C}^\infty(M, \operatorname{Hom}(E_1, E_2))[[\hbar]] \to \mathcal{C}^\infty(M, \operatorname{Hom}(E_1, E_3))[[\hbar]],
\end{equation}
which is still denoted by $\star$ by abuse of notations.

\subsection{Formal quantizability and degree}
\label{Subsection 4.3}
\quad\par
We first assign a weight to each element in $\mathcal{W} \otimes \operatorname{Hom}(E_1, E_2)$ by declaring that
\begin{itemize}
	\item $w^\alpha$ and elements in $\operatorname{Hom}(E_1, E_2)$ are of weight $0$, while
	\item $\overline{w}^\beta$ and $\hbar$ are of weight $1$.
\end{itemize}
This induces an increasing filtration
\begin{equation*}
	\mathcal{W}_{(0)} \otimes \operatorname{Hom}(E_1, E_2) \subset \mathcal{W}_{(1)} \otimes \operatorname{Hom}(E_1, E_2) \subset \cdots \mathcal{W}_{(r)} \otimes \operatorname{Hom}(E_1, E_2) \subset \cdots
\end{equation*}
of $\mathcal{W} \otimes \operatorname{Hom}(E_1, E_2)$, where $\mathcal{W}_{(r)}$ is the subbundle of elements of weight at most $r$ in $\mathcal{W}$ (note that $\mathcal{W}_{(0)} = \mathcal{W}_{\operatorname{cl}}^{1, 0} := \widehat{\operatorname{Sym}} T^\vee$), making $(\Omega^*(M, \mathcal{W} \otimes \operatorname{Hom}(E_1, E_2)), D^{E_1, E_2})$ a filtered cochain complex. We will call it the \emph{weight filtration}. 
\par
By (\ref{Equation 4.2}), we can see that the operator $\star$ preserves weight filtrations. In addition, observe that $\gamma \in \Omega^1(M, \underline{\mathcal{W}}_{\operatorname{cl}})$ and $I^{E_i} \in \Omega^1(M, \underline{\mathcal{W}}_{\operatorname{cl}} \otimes \operatorname{End}(E_i))$ for $i = 1, 2$, where $\underline{\mathcal{W}}_{\operatorname{cl}} := \widehat{\operatorname{Sym}} T^\vee \otimes \operatorname{Sym} \overline{T}^\vee$. Hence, we can easily verify that
\begin{equation*}
	\Omega^*(M, \underline{\mathcal{W}} \otimes \operatorname{Hom}(E_1, E_2))
\end{equation*}
is a subcomplex of $(\Omega^*(M, \mathcal{W} \otimes \operatorname{Hom}(E_1, E_2)), D^{E_1, E_2})$, where $\underline{\mathcal{W}} := \underline{\mathcal{W}}_{\operatorname{cl}} [\hbar]$.

\begin{definition}
	A \emph{formal quantizable morphism from} $E_1$ \emph{to} $E_2$ is an element
	\begin{equation*}
		\phi \in \mathcal{C}^\infty(M, \operatorname{Hom}(E_1, E_2))[[\hbar]]
	\end{equation*}
	for which the associated $D^{E_1, E_2}$-flat section $O_\phi$ lies in $\mathcal{C}^\infty(M, \underline{\mathcal{W}} \otimes \operatorname{Hom}(E_1, E_2))$.\par
	Let $r$ be a non-negative integer. A formal quantizable morphism $\phi$ from $E_1$ to $E_2$ is said to be \emph{of degree} $r$ if its associated $D^{E_1, E_2}$-flat section $O_\phi$ lies in $\mathcal{C}^\infty(M, \mathcal{W}_{(r)} \otimes \operatorname{Hom}(E_1, E_2))$.
\end{definition}

\begin{example}
	\label{Example 4.9}
	An element $\phi \in \mathcal{C}^\infty(M, \operatorname{Hom}(E_1, E_2))[[\hbar]]$ is formal quantizable of degree $0$ if and only if $\phi$ is a holomorphic section of $\operatorname{Hom}(E_1, E_2)$, in which case
	\begin{equation}
		O_\phi = \sum_{r=0}^\infty (\widetilde{\nabla}^{1, 0})^r \phi.
	\end{equation}
\end{example}

A direct consequence of Theorem \ref{Theorem 1.1} (to be proved in the next subsection) is that the enriched category $\mathsf{DQ}$ appeared in Theorem \ref{Theorem 1.1} has an enriched subcategory $\mathsf{DQ}_{\operatorname{qu}}$
\begin{itemize}
	\item which has the same objects as $\mathsf{DQ}$; and
	\item for any two objects $E_1, E_2$ of which, $\operatorname{Hom}_{\mathsf{DQ}_{\operatorname{qu}}}(E_1, E_2)$ is the sheaf of formal quantizable morphisms from $E_1$ to $E_2$.
\end{itemize}
In addition, $\operatorname{Hom}_{\mathsf{DQ}_{\operatorname{qu}}}(E_1, E_2)$ is a filtered left $\mathcal{O}_M$-module - a holomorphic function $f$ on $M$ acts on formal quantizable morphisms from $E_1$ to $E_2$ by composition with $f \operatorname{Id}_{E_2}$ and the filtration is given by the degrees of formal quantizable morphisms.

\begin{example}
	Suppose $E$ is a trivial Hermitian holomorphic line bundle over $M$. Then formal quantizable morphisms from $E$ to itself are exactly \emph{formal quantizable functions} on $M$, which are defined in \cite{ChaLeuLi2023}. We denote by $\mathcal{C}_{M, \operatorname{qu}}^\infty$ the sheaf of formal quantizable functions on $M$.
\end{example}

\subsection{Proof of Theorem \ref{Theorem 1.1}}
\label{Subsection 4.4}
\quad\par
Now, we prove our first main result, which states that we can quantize $(\mathsf{C}, \{\quad, \quad\})$, the classical category of Hermitian holomorphic vector bundles over $M$ equipped with covariantized Poisson brackets, so that the quantization we obtained is with separation of variables.
\begin{theorem}[$=$ Theorem \ref{Theorem 1.1}]
	Let $(M, \omega)$ be a K\"ahler manifold. Then there exists a deformation quantization $\mathsf{DQ}$ with separation of variables of $(\mathsf{C}, \{\quad, \quad\})$ such that
	\begin{itemize}
		\item for any objects $E_1, E_2, E_3$ in $\mathsf{DQ}$, the composition
		\begin{equation*}
			\operatorname{Hom}_{\mathsf{DQ}} (E_2, E_3) \otimes_{\mathbb{C}[[\hbar]]} \operatorname{Hom}_{\mathsf{DQ}} (E_1, E_2) \to \operatorname{Hom}_{\mathsf{DQ}} (E_1, E_3)
		\end{equation*}
		is given by $\star$ defined as in (\ref{Equation 4.9}).
		\item (\emph{degree preserving property}) for any objects $E_1, E_2$ in $\mathsf{DQ}$, open subset $U$ of $M$, $\phi \in \operatorname{Hom}_{\mathsf{DQ}} (E_1, E_2)(U)$ and $\psi \in \operatorname{Hom}_{\mathsf{DQ}} (E_2, E_3)(U)$, if $\phi, \psi$ are formal quantizable of degrees $r_1, r_2$ respectively, then $\psi \star \psi$ is formal quantizable of degree $r_1 + r_2$.
	\end{itemize}
\end{theorem}
\begin{proof}
	Associativity of $\star$ is easily deduced from that of the fibrewise anti-Wick product, those of fibrewise compositions of linear maps, and Proposition \ref{Proposition 4.7}. Once we have proved that the condition of separation of variables is satisfied, we will see that, for any objects $E_1, E_2, E_3$ in $\mathsf{DQ}$ and open subset $U$ of $M$, as $\operatorname{Id}_{E_2 \vert_U}$ is both holomorphic and anti-holomorphic,
	\begin{align*}
		\operatorname{Id}_{E_2 \vert_U} \star \phi = \phi, & \quad \text{for any } \phi \in \operatorname{Hom}_{\mathsf{C}} (E_1, E_2)(U)[[\hbar]],\\
		\psi \star \operatorname{Id}_{E_2 \vert_U} = \psi, & \quad \text{for any } \psi \in \operatorname{Hom}_{\mathsf{C}} (E_2, E_3)(U)[[\hbar]],
	\end{align*}
	whence $\mathsf{DQ}$ is a well defined enriched category.\par
	Now we compute the formal power series $\phi \star \phi$ in $\hbar$ up to first order for $\phi \in \operatorname{Hom}_{\mathsf{C}} (E_1, E_2)(U)$ and $\psi \in \operatorname{Hom}_{\mathsf{C}} (E_2, E_3)(U)$. Note that $O_\phi$ is uniquely determined by the following equality:
	\begin{align*}
		O_\phi = & \phi + \delta^{-1} \left( D^{E_1, E_2} + \delta \right) O_\phi\\
		= & \phi + \delta^{-1} \left( (\nabla + \nabla^{\operatorname{Hom}(E_1, E_2)}) O_\phi + \tfrac{1}{\hbar} [\widetilde{I}, O_\phi ]_\star + I^{E_2} \star O_\phi - O_\phi \star I^{E_1} \right).
	\end{align*}
	As $\delta^{-1} \left( D^{E_1, E_2} + \delta \right) O_\phi$ has zero $\operatorname{Hom}(E_1, E_2)[[\hbar]]$-component, the $\operatorname{Hom}(E_1, E_2)[[\hbar]]$-component of $O_\phi$ is thus $\phi$; as $\delta^{-1} \left( \tfrac{1}{\hbar} [\widetilde{I}, O_\phi ]_\star + I^{E_2} \star O_\phi - O_\phi \star I^{E_1} \right)$ has zero $T_\mathbb{C}^\vee \otimes \operatorname{Hom}(E_1, E_2)$-component, the $T_\mathbb{C}^\vee \otimes \operatorname{Hom}(E_1, E_2)$-component of $O_\phi$ is $\delta^{-1} \nabla^{\operatorname{Hom}(E_1, E_2)} \phi$. Similarly, the $\operatorname{Hom}(E_2, E_3)[[\hbar]]$- and $T_\mathbb{C}^\vee \otimes \operatorname{Hom}(E_2, E_3)$-components of $O_\psi$ are $\psi$ and $\delta^{-1} \nabla^{\operatorname{Hom}(E_2, E_3)} \psi$ respectively. Then we can see from (\ref{Equation 4.2}) that
	\begin{equation*}
		\psi \star \phi = \psi \phi + \hbar \omega^{\overline{\nu}\mu} (\nabla_{\partial_{\overline{z}^\nu}}^{\operatorname{Hom}(E_2, E_3)} \psi) (\nabla_{\partial_{z^\mu}}^{\operatorname{Hom}(E_1, E_2)} \phi) \pmod {\hbar^2}.
	\end{equation*}
	Therefore, the condition of classical limit is satisfied.\par
	Next, consider $f \in \mathcal{C}^\infty(U, \mathbb{C})$. Since $X_f = \omega^{\overline{\nu}\mu} \left( \tfrac{\partial f}{\partial \overline{z}^\nu} \tfrac{\partial}{\partial z^\mu} - \tfrac{\partial f}{\partial z^\mu} \tfrac{\partial}{\partial \overline{z}^\nu} \right)$,
	\begin{equation*}
		(f \operatorname{Id}_{E_2}) \star \phi - \phi \star (f\operatorname{Id}_{E_1}) = \hbar \nabla_{X_f}^{\operatorname{Hom}(E_1, E_2)} \phi = \{f, \phi\} \pmod {\hbar^2}. 
	\end{equation*}
	Hence, the condition of semi-classical limit is satisfied.\par
	We then consider $\phi \in \operatorname{Hom}_{\mathsf{C}} (E_1, E_2)(U)[[\hbar]]$ and $\psi \in \operatorname{Hom}_{\mathsf{C}} (E_2, E_3)(U)[[\hbar]]$. To compute
	\begin{equation*}
		\psi \star \phi = \pi_0( O_\psi \star O_\phi )
	\end{equation*}
	we only need the $\mathcal{W}^{0, 1} \otimes \operatorname{Hom}(E_2, E_3)$-component of $O_\psi$ and the $\mathcal{W}^{1, 0} \otimes \operatorname{Hom}(E_1, E_2)$-component of $O_\phi$. If $\psi$ is holomorphic, then $O_\psi$ is a section of $\mathcal{W}_{\operatorname{cl}}^{1, 0} \otimes \operatorname{Hom}(E_2, E_3)$ by Example \ref{Example 4.9} and hence the $\mathcal{W}^{0, 1} \otimes \operatorname{Hom}(E_2, E_3)$-component of $O_\psi$ is $\psi$; if $\phi$ is anti-holomorphic, then the $\mathcal{W}^{1, 0} \otimes \operatorname{Hom}(E_1, E_2)$-component of $O_\phi$ is $\sum_{r=0}^\infty (\widetilde{\nabla}^{1, 0})^r \phi = \phi$ by Remark \ref{Remark 4.6}. In any one of the above cases, we see from (\ref{Equation 4.2}) again that $\psi \star \phi = \psi \phi$. Thus, the condition of separation of variables is satisfied.\par
	Finally, we can also easily deduce from (\ref{Equation 4.2}) that $\mathsf{DQ}$ has the degree preserving property.
\end{proof}

\subsection{Non-formal quantizable morphisms of Hermitian holomorphic vector bundles}
\label{Subsection 4.5}
\quad\par
Let $E_1, E_2$ be Hermitian holomorphic vector bundles over $M$ and $k$ be a non-zero complex number. We can evaluate $D^{E_1, E_2}$ at $\hbar = \tfrac{\sqrt{-1}}{k}$ without convergence issues and obtain a non-formal flat connection
\begin{equation}
	D_k^{E_1, E_2} = \nabla + \nabla^{\operatorname{Hom}(E_1, E_2)} + \tfrac{k}{\sqrt{-1}} [\gamma, \quad]_{\star_k} + I^{E_2} \star_k - \star_k I^{E_1}
\end{equation}
on $\underline{\mathcal{W}}_{\operatorname{cl}} \otimes \operatorname{Hom}(E_1, E_2)$. Here, recall that $\underline{\mathcal{W}}_{\operatorname{cl}} = \widehat{\operatorname{Sym}} T^\vee \otimes \operatorname{Sym} \overline{T}^\vee$, and we add a subscript $k$ in a symbol to denote its evaluation at $\hbar = \tfrac{\sqrt{-1}}{k}$.\par
We then have the following generalization of Definition 2.20 in \cite{ChaLeuLi2023}.
\begin{definition}
	Let $E_1, E_2$ be Hermitian holomorphic vector bundles over $M$ and $k \in \mathbb{C}$ be non-zero. A \emph{level}-$k$ \emph{quantizable morphism from} $E_1$ \emph{to} $E_2$ is a $D_k^{E_1, E_2}$-flat section of $\underline{\mathcal{W}}_{\operatorname{cl}} \otimes \operatorname{Hom}(E_1, E_2)$.
\end{definition}

\begin{remark}
	In general, a level-$k$ quantizable morphism $O$ from $E_1$ to $E_2$ is not determined by $\pi_0(O)$ -- there can be two distinct level-$k$ quantizable morphisms $O, O'$ from $E_1$ to $E_2$ such that $\pi_0(O) = \pi_0(O') \in \mathcal{C}^\infty(M, \operatorname{Hom}(E_1, E_2))$ (see Example 2.25 in \cite{ChaLeuLi2023}).
\end{remark}

Now we can construct the following category $\mathsf{DQ}_{\operatorname{qu}, k}$, enriched over $\mathsf{Sh}(M)$, as follows:
\begin{itemize}
	\item objects in $\mathsf{DQ}_{\operatorname{qu}, k}$ are Hermitian holomorphic vector bundles over $M$;
	\item for any two objects $E_1, E_2$ in $\mathsf{DQ}_{\operatorname{qu}, k}$, $\operatorname{Hom}_{\mathsf{DQ}_{\operatorname{qu}, k}}(E_1, E_2)$ is the sheaf of level-$k$ quantizable morphisms from $E_1$ to $E_2$;
	\item the composition in $\mathsf{DQ}_{\operatorname{qu}, k}$ is given by $\star_k$.
\end{itemize}
Regarding $\mathsf{DQ}_{\operatorname{qu}}$ as a category enriched over $\mathsf{Sh}(M)$, there is naturally an enriched functor
\begin{equation*}
	\operatorname{ev}_k: \mathsf{DQ}_{\operatorname{qu}} \to \mathsf{DQ}_{\operatorname{qu}, k}
\end{equation*}
given as follows:
\begin{itemize}
	\item for any object $E$ in $\mathsf{DQ}_{\operatorname{qu}}$, $\operatorname{ev}_k(E) = E$;
	\item for any objects $E_1, E_2$ in $\mathsf{DQ}_{\operatorname{qu}}$,
	\begin{equation}
		\label{Equation 4.12}
		\operatorname{ev}_k: \operatorname{Hom}_{\mathsf{DQ}_{\operatorname{qu}}}(E_1, E_2) \to \operatorname{Hom}_{\mathsf{DQ}_{\operatorname{qu}, k}}(E_1, E_2)
	\end{equation}
	is given by $\phi \mapsto O_\phi$ and then taking evaluations at $\hbar = \tfrac{\sqrt{-1}}{k}$.
\end{itemize}

By Theorem \ref{Theorem 1.1} again, $\operatorname{Hom}_{\mathsf{DQ}_{\operatorname{qu}, k}}(E_1, E_2)$ is a filtered left $\mathcal{O}_M$-module -- a holomorphic function $f$ on $M$ acts on level-$k$ quantizable morphisms from $E_1$ to $E_2$ via $O_{f \operatorname{Id}_{E_2}} \star_k$ and the left $\mathcal{O}_M$-module $\operatorname{Hom}_{\mathsf{DQ}_{\operatorname{qu}, k}}(E_1, E_2)$ inherits the weight filtration
\begin{equation*}
	\operatorname{F}_0\operatorname{Hom}_{\mathsf{DQ}_{\operatorname{qu}, k}}(E_1, E_2) \hookrightarrow \operatorname{F}_1\operatorname{Hom}_{\mathsf{DQ}_{\operatorname{qu}, k}}(E_1, E_2) \hookrightarrow \operatorname{F}_2\operatorname{Hom}_{\mathsf{DQ}_{\operatorname{qu}, k}}(E_1, E_2) \hookrightarrow \cdots.
\end{equation*}
The map (\ref{Equation 4.12}) is indeed a morphism of filtered left $\mathcal{O}_M$-modules, and it restricts to an isomorphism  of left $\mathcal{O}_M$-modules:
\begin{equation*}
	\underline{\operatorname{Hom}}_{\mathcal{O}_M}(\mathcal{E}_1, \mathcal{E}_2) \cong \operatorname{F}_0 \operatorname{Hom}_{\mathsf{DQ}_{\operatorname{qu}, k}}(E_1, E_2).
\end{equation*}

\begin{example}
	Suppose $E$ is a trivial Hermitian holomorphic line bundle over $M$. Then level-$k$ quantizable morphisms from $E$ to itself are exactly \emph{level-$k$ quantizable functions} on $M$, which are again defined in \cite{ChaLeuLi2023}. We denote by $\mathcal{C}_{M, \operatorname{qu}, k}^\infty$ the sheaf of level-$k$ quantizable functions on $M$.
\end{example}

\section{Actions of the quantum category of Hermitian holomorphic vector bundles}
\label{Section 5}
Assume the K\"ahler manifold $(M, \omega)$ is prequantizable and pick a prequantum line bundle $L$. Instead of a non-zero complex number, we suppose $k \in \mathbb{Z}^+$ is a positive integer. In \cite{ChaLeuLi2023}, Chan-Leung-Li proved that the sheaf $\mathcal{C}_{M, \operatorname{qu}, k}^\infty$ of level-$k$ quantizable functions on $M$ is isomorphic to the sheaf $\mathcal{D}(\mathcal{L}^{\otimes k}, \mathcal{L}^{\otimes k})$ of holomorphic differential operators from $L^{\otimes k}$ to itself (see also \cite{Yau2024}).\par
In this section, we will prove a categorical generalization (Theorem \ref{Theorem 1.2}) of the above result. In Subsection \ref{Subsection 5.1}, we will discuss Berezin-Toeplitz quantization as a motivation of Theorem \ref{Theorem 1.2}. In subsection \ref{Subsection 5.2}, we will define the category $\mathsf{GQ}$ appeared in Theorem \ref{Theorem 1.2}. Then in Subsection \ref{Subsection 5.3}, we will construct the fibrewise Bargmann-Fock action, which is a key ingredient in the proof of this theorem. Eventually, Theorem \ref{Theorem 1.2} will be proved in Subsection \ref{Subsection 5.4}.

\subsection{A motivation of Theorem \ref{Theorem 1.2}: Berezin-Toeplitz quantization}
\label{Subsection 5.1}
\quad\par
In this subsection, suppose $M$ is compact. Geometric quantization provides a recipe to construct a Hilbert space mathematically realizing the space of quantum states for quantum mechanics on the phase space $(M, \omega)$, which is $H^0(M, L^{\otimes k})$ when K\"ahler polarization is chosen. Physicists expect that it is not only a Hilbert space, but also a module over deformation quantization of $(M, \omega)$.\par
Another quantization scheme, known as \emph{Berezin-Toeplitz quantization}, is a microlocal-theoretic approach to construct such a module structure. Upon Berezin-Toeplitz quantization, every smooth function $f \in \mathcal{C}^\infty(M, \mathbb{C})$ is assigned to its family of \emph{Toeplitz operators} parametrized by $k \in \mathbb{Z}^+$:
\begin{equation*}
	T_{f, k} = \Pi \circ f \circ \Pi: H^0(M, L^{\otimes k}) \to H^0(M, L^{\otimes k}).
\end{equation*}
Here, $\Pi: L^2(M, L^{\otimes k}) \to H^0(M, L^{\otimes k})$ is the orthogonal projection and multiplication by $f$ is also denoted by $f$ by abuse of notations. Note that this construction relies on the $L^2$-inner product on $L^2(M, L^{\otimes k})$ and hence requires compactness of $M$. Schlichenmaier \cite{Sch2000} proved, via microlocal analysis, that the asymptotic expansion of $T_{f, k} \circ T_{g, k}$ (as $k \to \infty$) for any $f, g \in \mathcal{C}^\infty(M, \mathbb{C})$ defines a star product $\star_{\operatorname{BT}}$ on $(M, \omega)$. Roughly speaking, Toeplitz operators provide an `asymptotic' (but not an honest) module structure on $H^0(M, L^{\otimes k})$.\par
When geometric quantization of $(M, \omega)$ is coupled with a Hermitian holomorphic vector bundle $E$ over $M$, the concerned Hilbert space becomes the kernel of the $\operatorname{spin}^{\operatorname{c}}$-Dirac operator $\slashed{D} = \sqrt{2}(\overline{\partial} + \overline{\partial}^*)$ on $\Omega^{0, *}(M, E \otimes L^{\otimes k})$, identified with $H^*(M, E \otimes L^{\otimes k})$ by Hodge theory. In this case, \emph{Toeplitz operators} associated with a smooth section $\phi \in \mathcal{C}^\infty(M, \operatorname{End}(E))$ are given by
\begin{equation*}
	T_{\phi, k}^E = \Pi^E \circ \phi \circ \Pi^E: H^*(M, E \otimes L^{\otimes k}) \to H^*(M, E \otimes L^{\otimes k}),
\end{equation*}
where $\Pi^E: L^2 \left(M, \bigwedge \overline{T}^\vee \otimes E \otimes L^{\otimes k}\right) \to \ker \slashed{D} \cong H^*(M, E \otimes L^{\otimes k})$ is the orthogonal projection (see \cite{MaMar2008}). In \cite{MaMar2012}, Ma-Marinescu proved a generalized result that the asymptotic expansion of $T_{\psi, k}^E \circ T_{\phi, k}^E$'s for any $\phi, \psi \in \mathcal{C}^\infty(M, \operatorname{End}(E))$ defines a $\operatorname{End}(E)$-valued star product on $(M, \omega)$.\par
Indeed, for each pair of Hermitian holomorphic vector bundles $E_1, E_2$ over $M$, any smooth section $\phi \in \mathcal{C}^\infty(M, \operatorname{Hom}(E_1, E_2))$ can still be assigned to an analogous family of operators parametrized by $k \in \mathbb{Z}^+$:
\begin{equation*}
	T_{\phi, k}^{E_1, E_2} = \Pi^{E_2} \circ \phi \circ \Pi^{E_1}: H^*(M, E_1 \otimes L^{\otimes k}) \to H^*(M, E_2 \otimes L^{\otimes k}).
\end{equation*}
In \cite{AdaIshKan2023}, Adachi-Ishiki-Kanno investigated the asymptotic expansion of $T_{\psi, k}^{E_2, E_3} \circ T_{\phi, k}^{E_1, E_2}$ for any Hermitian holomorphic vector bundles $E_1, E_2, E_3$ over $M$ and $\phi \in \mathcal{C}^\infty(M, \operatorname{Hom}(E_1, E_2))$ and $\psi \in \mathcal{C}^\infty(M, \operatorname{Hom}(E_2, E_3))$. Their results stated in (2.10), (2.12) and (2.14) in \cite{AdaIshKan2023} can be reformulated as that the above asymptotic expansion defines a deformation quantization $\mathsf{DQ}_{\operatorname{BT}}$ of $(\mathsf{C}, \{\quad, \quad\})$ in the sense of Definition \ref{Definition 2.5}. Heuristically, we can regard the family of operators $T_{\phi, k}^{E_1, E_2}$'s as an asymptotic action of $\mathsf{DQ}_{\operatorname{BT}}$ on the family of spaces $H^*(M, E \otimes L^{\otimes k})$'s.

\subsection{Categorification of geometric quantization}
\label{Subsection 5.2}
\quad\par
In this paper, we are only interested in operators
\begin{equation*}
	H^*(M, E_1 \otimes L^{\otimes k}) \to H^*(M, E_2 \otimes L^{\otimes k})
\end{equation*}
which are holomorphic differential operators. We will take Chan-Leung-Li's sheaf-theoretic approach \cite{ChaLeuLi2023} to send quantizable morphisms to holomorphic differential operators via Bargmann-Fock actions. Under this approach, there is no need to assume that $M$ is compact. We will also see from Theorem \ref{Theorem 1.2} that there is an honest action of $\mathsf{DQ}_{\operatorname{qu}, k}$ on the family of sheaves $\mathcal{E} \otimes \mathcal{L}^{\otimes k}$'s, or more precisely, an enriched functor from $\mathsf{DQ}_{\operatorname{qu}, k}$ to the enriched category $\mathsf{GQ}$ defined as follows.

\begin{definition}
	We define the category $\mathsf{GQ}$, enriched over $\mathsf{Sh}(M)$, as follows:
	\begin{itemize}
		\item objects in $\mathsf{GQ}$ are holomorphic vector bundles over $M$;
		\item for any objects $E_1, E_2$ in $\mathsf{GQ}$,
		\begin{equation*}
			\operatorname{Hom}_{\mathsf{GQ}}(E_1, E_2) = \mathcal{D}(\mathcal{E}_1, \mathcal{E}_2);
		\end{equation*}
		\item for any objects $E_1, E_2, E_2$ in $\mathsf{GQ}$,
		\begin{equation*}
			\operatorname{Hom}_{\mathsf{GQ}}(E_2, E_3) \otimes_\mathbb{C} \operatorname{Hom}_{\mathsf{GQ}}(E_1, E_2) \to \operatorname{Hom}_{\mathsf{GQ}}(E_1, E_3)
		\end{equation*}
		is the usual composition of holomorphic differential operators.
	\end{itemize}
\end{definition}

\begin{remark}
	Unlike $\mathsf{DQ}$, we do not require objects in $\mathsf{GQ}$ to be equipped with Hermitian metrics. We will see that Hermitian metrics are encoded in the category $\mathsf{DQ}_{\operatorname{qu}, k}$ and a functor $\mathsf{DQ}_{\operatorname{qu}, k} \to \mathsf{GQ}$.
\end{remark}

Note that $\operatorname{Hom}_{\mathsf{GQ}}(E_1, E_2) = \mathcal{D}(\mathcal{E}_1, \mathcal{E}_2)$ is a filtered $\mathcal{O}_M$-module equipped with the natural filtration induced by the orders of holomorphic differential operators. The composition in $\mathsf{GQ}$ preserves the filtrations on morphism sheaves, but not the $\mathcal{O}_M$-module structures.\par
Denote by $\mathsf{GQ}^{\operatorname{op}}$ the opposite enriched category of $\mathsf{GQ}$. There is an enriched functor
\begin{equation*}
	\tau: \mathsf{GQ} \to \mathsf{GQ}^{\operatorname{op}}
\end{equation*}
defined as follows:
\begin{itemize}
	\item for any object $E$ in $\mathsf{GQ}$, $\tau(E) = \operatorname{Hom}(E, K)$, where $K$ is the canonical bundle of $M$;
	\item for any objects $E_1, E_2$ in $\mathsf{GQ}$,
	\begin{equation*}
		\tau: \operatorname{Hom}_{\mathsf{GQ}}(E_1, E_2) \to \operatorname{Hom}_{\mathsf{GQ}}(\operatorname{Hom}(E_2, K), \operatorname{Hom}(E_1, K))
	\end{equation*}
	is given by sending a holomorphic differential operator to its \emph{holomorphic transposed operator} (see Appendix \ref{Appendix B}).
\end{itemize}
Moreover, $\tau$ yields an equivalence of the enriched categories $\mathsf{GQ}$ and $\mathsf{GQ}^{\operatorname{op}}$.

\subsection{The fibrewise Bargmann-Fock action and Kapranov's connection}
\label{Subsection 5.3}
\quad\par
We first define the fibrewise Bargmann-Fock action $\circledast_k$ (in the anti-Wick ordering). For any Hermitian holomorphic vector bundle $E$ over $M$, let $\mathscr{T}_k(E) = E \otimes L^{\otimes k}$. 

\begin{definition}
	Let $E_1, E_2$ be Hermitian holomorphic vector bundles over $M$ and $k \in \mathbb{Z}^+$. Define a $\mathcal{C}^\infty(M, \mathbb{C})$-bilinear map
	\begin{equation}
		\label{Equation 5.1}
		\mathcal{C}^\infty(M, \underline{\mathcal{W}}_{\operatorname{cl}} \otimes \operatorname{Hom}(E_1, E_2)) \times \mathcal{C}^\infty(M, \mathcal{W}_{\operatorname{cl}}^{1, 0} \otimes \mathscr{T}_k(E_1)) \to \mathcal{C}^\infty(M, \mathcal{W}_{\operatorname{cl}}^{1, 0} \otimes \mathscr{T}_k(E_2))
	\end{equation}
	as follows. We write the above map as $(a, s) \mapsto a \circledast_k s$. For $a = w^{\mu_1} \cdots w^{\mu_p} \overline{w}^{\nu_1} \cdots \overline{w}^{\nu_q} \otimes A$ with $A \in \mathcal{C}^\infty(M, \operatorname{Hom}(E_1, E_2))$ and $s \in \mathcal{C}^\infty(M, \mathcal{W}_{\operatorname{cl}}^{1, 0} \otimes \mathscr{T}_k(E_1))$,
	\begin{equation}
		\label{Equation 5.2}
		a \circledast_k s = A \left( \left( \frac{\sqrt{-1}}{k} \right)^q \omega^{\overline{\nu}_1\lambda_1} \cdots \omega^{\overline{\nu}_q\lambda_q} w^{\mu_1} \cdots w^{\mu_p} \frac{\partial^q}{\partial w^{\lambda_1} \cdots \partial w^{\lambda_q}} (s) \right).
	\end{equation}
	In (\ref{Equation 5.2}), $A$ is regarded as a section of $\operatorname{Hom}(\mathcal{W}_{\operatorname{cl}}^{1, 0} \otimes \mathscr{T}_k(E_1), \mathcal{W}_{\operatorname{cl}}^{1, 0} \otimes \mathscr{T}_k(E_2))$.
\end{definition}

Note that if we have three arbitrary Hermitian holomorphic vector bundles $E_1, E_2, E_3$ over $M$, $a \in \mathcal{C}^\infty(M, \underline{\mathcal{W}}_{\operatorname{cl}} \otimes \operatorname{Hom}(E_2, E_3))$, $b \in \mathcal{C}^\infty(M, \underline{\mathcal{W}}_{\operatorname{cl}} \otimes \operatorname{Hom}(E_1, E_2))$ and $s \in \mathcal{C}^\infty(M, \mathcal{W}_{\operatorname{cl}}^{1, 0} \otimes \mathscr{T}_k(E_1))$, then
\begin{equation}
	\label{Equation 5.3}
	a \circledast_k (b \circledast_k s) = (a \star_k b) \circledast_k s.
\end{equation}

When $E_1 = E_2 = E_3$ is a trivial Hermitian holomorphic line bundle $E$ over $M$, Chan-Leung-Li \cite{ChaLeuLi2023} showed that the operation $\circledast_k$ descends to a morphism of sheaves on $M$:
\begin{equation*}
	\mathcal{C}_{M, \operatorname{qu}, k}^\infty \times \mathcal{L}^{\otimes k} \to \mathcal{L}^{\otimes k}.
\end{equation*}
They proved it by constructing a flat connection on $\mathcal{W}_{\operatorname{cl}}^{1, 0} \otimes L^{\otimes k}$ which is compatible with the flat connection $D_k^{E, E}$ on $\mathcal{C}^\infty(M, \underline{\mathcal{W}}_{\operatorname{cl}})$ and the operation $\circledast_k$. This connection is given by
\begin{equation}
	\nabla + \nabla^{L^{\otimes k}} + \tfrac{k}{\sqrt{-1}} \gamma \circledast_k
\end{equation}
(see also \cite{Yau2024} for a similar construction when $M$ is spin and $E_1 = E_2 = E_3$ is a square root of the canonical bundle of $M$). In the rest of this subsection, we will clarify that Chan-Leung-Li's flat connection is indeed the Kapranov's connection $D_{\operatorname{Kap}}^{L^{\otimes k}}$ given as in (\ref{Equation 3.4}).

\begin{proposition}
	\label{Proposition 5.4}
	Let $E$ be a Hermitian holomorphic vector bundle over $M$. Then
	\begin{equation*}
		D_{\operatorname{Kap}}^{E \otimes L^{\otimes k}} = \nabla + \nabla^{E \otimes L^{\otimes k}} + \tfrac{k}{\sqrt{-1}} \gamma \circledast_k + I^E, \quad \text{on } \mathcal{W}_{\operatorname{cl}}^{1, 0} \otimes E \otimes L^{\otimes k}.
	\end{equation*}
\end{proposition}

\begin{lemma}
	\label{Lemma 5.5}
	On $\mathcal{W}_{\operatorname{cl}}^{1, 0} \otimes E \otimes L^{\otimes k}$, $\tfrac{k}{\sqrt{-1}} ( (\delta^{0, 1})^{-1} \omega ) \circledast_k = -\delta^{1, 0}$.
\end{lemma}
\begin{proof}
	Note that $(\delta^{0, 1})^{-1} \omega = -\omega_{\alpha\overline{\beta}} \overline{w}^\beta dz^\alpha$. Then
	\begin{equation*}
		\frac{k}{\sqrt{-1}} ( (\delta^{0, 1})^{-1} \omega ) \circledast_k = -\omega^{\overline{\beta}\lambda} \omega_{\alpha\overline{\beta}} dz^\alpha \wedge \frac{\partial}{\partial w^\lambda} = -dz^\alpha \wedge \frac{\partial}{\partial w^\alpha} = -\delta^{1, 0}.
	\end{equation*}
\end{proof}

\begin{lemma}
	\label{Lemma 5.6}
	$I^L = \tfrac{1}{\sqrt{-1}} (\delta^{1, 0})^{-1}\omega$.
\end{lemma}
\begin{proof}
	As the curvature of $L$ is $\tfrac{1}{\sqrt{-1}} \omega$, $I_{(1)}^L = \tfrac{1}{\sqrt{-1}} (\delta^{1, 0})^{-1}\omega$. Locally, $(\delta^{1, 0})^{-1}\omega = \omega_{\alpha\overline{\beta}} w^\alpha d\overline{z}^\beta$. Then
	\begin{align*}
		\nabla^{1, 0} (\delta^{1, 0})^{-1}\omega = & \frac{1}{2} \left( \frac{\partial \omega_{\alpha\overline{\beta}}}{\partial z^\gamma} - \omega_{\lambda\overline{\beta}} \Gamma_{\gamma\alpha}^\lambda \right) w^\alpha w^\gamma d\overline{z}^\beta = 0
	\end{align*}
	by a standard identity for Christoffel symbols $\Gamma_{\alpha\beta}^\gamma$ of $M$. Thus, $I_{(r)}^L = 0$ for all $r > 1$.
\end{proof}

\begin{lemma}
	\label{Lemma 5.7}
	On $\mathcal{W}_{\operatorname{cl}}^{1, 0} \otimes E \otimes L^{\otimes k}$, $\tfrac{k}{\sqrt{-1}} \widetilde{I} \circledast_k = I$.
\end{lemma}
\begin{proof}
	Fix $r \geq 2$. Write $\widetilde{I}_{(r)}$ locally as $\widetilde{I}_{(r)} = \omega_{\eta\overline{\nu}} F_{\mu_1, ..., \mu_r, \overline{\beta}}^\eta d\overline{z}^\beta \otimes w^{\mu_1} \cdots w^{\mu_r} \overline{w}^\nu$. Then for all $s \in \mathcal{C}^\infty(M, \mathcal{W}_{\operatorname{cl}}^{1, 0})$, we have
	\begin{equation*}
		\tfrac{k}{\sqrt{-1}} \widetilde{I}_{(r)} \circledast_k s = \omega^{\overline{\mu}\lambda} \omega_{\eta\overline{\mu}} F_{\mu_1, ..., \mu_r, \overline{\beta}}^\eta d\overline{z}^\beta \otimes w^{\mu_1} \cdots w^{\mu_r} \frac{\partial s}{\partial w^\lambda} = I_{(r)} s.
	\end{equation*}
\end{proof}

\begin{proof}[\myproof{Proposition}{\ref{Proposition 5.4}}]
	Recall from (\ref{Equation 3.4}) that
	\begin{equation*}
		D_{\operatorname{Kap}}^{E \otimes L^{\otimes k}} = \nabla + \nabla^{E \otimes L^{\otimes k}} - \delta^{1, 0} + I + I^{E \otimes L^{\otimes k}}.
	\end{equation*}
	It suffices to show that $\tfrac{k}{\sqrt{-1}} \gamma \circledast_k = -\delta^{1, 0} + I + I^{L^{\otimes k}}$. By Lemmas \ref{Lemma 5.5}, \ref{Lemma 5.6} and \ref{Lemma 5.7}, we are done.
\end{proof}

\subsection{Proof of Theorem \ref{Theorem 1.2}}
\label{Subsection 5.4}
\quad\par
Let us recall the theorem.

\begin{theorem}[$=$ Theorem \ref{Theorem 1.2}]
	Let $(M, \omega)$ be a prequantizable K\"ahler manifold with a prequantum line bundle $L$ and $k \in \mathbb{Z}^+$. Then there exists an enriched functor 
	\begin{equation*}
		\mathscr{T}_k: \mathsf{DQ}_{\operatorname{qu}, k} \to \mathsf{GQ}
	\end{equation*}
	such that
	\begin{enumerate}
		\item for any object $E$ in $\mathsf{DQ}_{\operatorname{qu}, k}$, $\mathscr{T}_k(E) = E \otimes L^{\otimes k}$;
		\item for any objects $E_1, E_2$ in $\mathsf{DQ}_{\operatorname{qu}, k}$,
		\begin{equation*}
			\mathscr{T}_k: \operatorname{Hom}_{\mathsf{DQ}_{\operatorname{qu}, k}}(E_1, E_2) \to \operatorname{Hom}_\mathsf{GQ}(\mathscr{T}_k(E_1), \mathscr{T}_k(E_2))
		\end{equation*}
		is an isomorphism of filtered left $\mathcal{O}_M$-modules;
		\item $\mathscr{T}_k$ yields an equivalence of categories enriched over the monoidal category $\mathsf{Sh}(M)$ of sheaves of $\mathbb{C}$-vector spaces on $M$.
	\end{enumerate}
\end{theorem}

Before proving Theorem \ref{Theorem 1.2}, we need some preparations. Observe that for any two Hermitian holomorphic vector bundles $E_1, E_2$ over $M$, the map (\ref{Equation 5.1}) naturally induces an isomorphism of vector bundles (of infinite rank) over $M$:
\begin{equation}
	\label{Equation 5.5}
	\underline{\mathcal{W}}_{\operatorname{cl}} \otimes \operatorname{Hom}(E_1, E_2) \cong \operatorname{Hom}^{<\infty}( \mathcal{W}_{\operatorname{cl}}^{1, 0} \otimes \mathscr{T}_k(E_1), \mathcal{W}_{\operatorname{cl}}^{1, 0} \otimes \mathscr{T}_k(E_2) ),
\end{equation}
the bundle on the right hand side of which is defined in Subsection \ref{Subsection 3.4}. Now we prove a key lemma.

\begin{lemma}
	\label{Lemma 5.9}
	Under the identification (\ref{Equation 5.5}),
	\begin{enumerate}
		\item $D_k^{E_1, E_2}$ coincides with the Kapranov's connection $D_{\operatorname{Kap}}^{\mathscr{T}_k(E_1), \mathscr{T}_k(E_2)}$;
		\item the operation $\star_k$ in (\ref{Equation 4.5}) evaluated at $\hbar = \tfrac{\sqrt{-1}}{k}$ and restricted on $\underline{\mathcal{W}}_{\operatorname{cl}} \otimes \operatorname{Hom}(E_i, E_j)$'s coincides with the fibrewise composition of $\operatorname{Hom}( \mathcal{W}_{\operatorname{cl}}^{1, 0} \otimes \mathscr{T}_k(E_i), \mathcal{W}_{\operatorname{cl}}^{1, 0} \otimes \mathscr{T}_k(E_j) )$'s.
	\end{enumerate}
\end{lemma}
\begin{proof}
	The second condition holds due to (\ref{Equation 5.3}). Now we compare
	\begin{equation*}
		D_k^{E_1, E_2} = \nabla + \nabla^{\operatorname{Hom}(E_1, E_2)} - \delta^{1, 0} + \tfrac{k}{\sqrt{-1}} [(\delta^{1, 0})^{-1} \omega + \widetilde{I}, \quad]_{\star_k} + I^{E_2} \star_k - \star_k I^{E_1}
	\end{equation*}
	with the formula of $D_{\operatorname{Kap}}^{\mathscr{T}_k(E_1), \mathscr{T}_k(E_2)}$ given as in (\ref{Equation 3.7}):
	\begin{equation*}
		D_{\operatorname{Kap}}^{\mathscr{T}_k(E_1), \mathscr{T}_k(E_2)} = \nabla + \nabla^{\operatorname{Hom}(\mathscr{T}_k(E_1), \mathscr{T}_k(E_2))} - \delta^{1, 0} + (I + I^{\mathscr{T}_k(E_2)}) \circ - \circ (I + I^{\mathscr{T}_k(E_1)}).
	\end{equation*}
	Since $\nabla \omega = 0$, (\ref{Equation 5.5}) identifies the connection $\nabla + \nabla^{\operatorname{Hom}(E_1, E_2)}$ with $\nabla + \nabla^{\operatorname{Hom}(\mathscr{T}_k(E_1), \mathscr{T}_k(E_2))}$. It remains to show that $\tfrac{k}{\sqrt{-1}} [(\delta^{1, 0})^{-1} \omega + \widetilde{I}, \quad]_{\star_k}$ is identified with $(I + I^{L^{\otimes k}}) \circ - \circ (I + I^{L^{\otimes k}})$. Indeed, this has been already verified by Lemmas \ref{Lemma 5.6} and \ref{Lemma 5.7} together with the property (\ref{Equation 5.3}).
\end{proof}

\begin{proof}[\myproof{Theorem}{\ref{Theorem 1.2}}]
	Define $\mathscr{T}_k(E) = E \otimes L^{\otimes k}$ for any object $E$ in $\mathsf{DQ}_{\operatorname{qu}, k}$. Now suppose $E_1, E_2$ are objects in $\mathsf{DQ}_{\operatorname{qu}, k}$. By Lemma \ref{Lemma 5.9}, we see that $\circledast_k$ induces a cochain isomorphism
	\begin{align*}
		& (\Omega^*(M, \underline{\mathcal{W}}_{\operatorname{cl}} \otimes \operatorname{Hom}(E_1, E_2)), D_k^{E_1, E_2})\\
		\cong & (\Omega^*(M, \operatorname{Hom}^{<\infty}( \mathcal{W}_{\operatorname{cl}}^{1, 0} \otimes \mathscr{T}_k(E_1), \mathcal{W}_{\operatorname{cl}}^{1, 0} \otimes \mathscr{T}_k(E_2) )), D_{\operatorname{Kap}}^{\mathscr{T}_k(E_1), \mathscr{T}_k(E_2)}).
	\end{align*}
	Then by Proposition \ref{Proposition 3.11}, we obtain an isomorphism of $\mathbb{C}$-linear sheaves
	\begin{equation}
		\label{Equation 5.6}
		\mathscr{T}_k: \operatorname{Hom}_{\mathsf{DQ}_{\operatorname{qu}, k}}(E_1, E_2) \to \mathcal{D}(\mathcal{E}_1 \otimes \mathcal{L}^{\otimes k}, \mathcal{E}_2 \otimes \mathcal{L}^{\otimes k}) = \operatorname{Hom}_{\mathsf{GQ}}( \mathscr{T}_k(E_1), \mathscr{T}_k(E_2) )
	\end{equation}
	For any object $E$ in $\mathsf{DQ}_{\operatorname{qu}, k}$, $\mathscr{T}_k(\operatorname{Id}_E) = \operatorname{Id}_{\mathscr{T}_k(E)}$ by (\ref{Equation 5.2}). Since $\star_k$ is identified with $\circ$ via (\ref{Equation 5.5}) by Lemma \ref{Lemma 5.9}, it is the usual composition of holomorphic differential operators when restricted on $\operatorname{Hom}_{\mathsf{DQ}_{\operatorname{qu}, k}}(E_i, E_j)$'s. In other words, $\mathscr{T}_k$ is compatible with compositions in $\mathsf{DQ}_{\operatorname{qu}, k}$ and $\mathsf{GQ}$. Therefore, $\mathscr{T}_k: \mathsf{DQ}_{\operatorname{qu}, k} \to \mathsf{GQ}$ is an enriched functor. It is obvious that $\mathscr{T}_k$ is essentially surjective, hence yields an equivalence of categories enriched over $\mathsf{Sh}(M)$.\par
	Observe from (\ref{Equation 5.2}) and Proposition \ref{Proposition 3.11} that the isomorphism (\ref{Equation 5.6}) preserves the filtrations. It remains to check that the isomorphism (\ref{Equation 5.6}) preserves the left $\mathcal{O}_M$-module structures. Consider any holomorphic function $f$ on $M$. By Example \ref{Example 4.9}, $O_{f\operatorname{Id}_{E_2}} = \sum_{r=0}^\infty (\widetilde{\nabla}^{1, 0})^r (f\operatorname{Id}_{E_2})$ is a section of $\mathcal{W}_{\operatorname{cl}}^{1, 0} \otimes \operatorname{Hom}(E_2, E_2)$. We can then see from (\ref{Equation 5.2}) that $\mathscr{T}_k(O_{f \operatorname{Id}_{E_2}}) = f\operatorname{Id}_{\mathscr{T}_k(E_2)}$. Therefore, for any $\phi \in \operatorname{Hom}_{\mathsf{DQ}_{\operatorname{qu}, k}}(E_1, E_2)(M)$,
	\begin{equation*}
		\mathscr{T}_k((O_{f \operatorname{Id}_{E_2}}) \star_k \phi) = \mathscr{T}_k(O_{f \operatorname{Id}_{E_2}}) \circ \mathscr{T}_k(\phi) = (f\operatorname{Id}_{\mathscr{T}_k(E_2)}) \circ \mathscr{T}_k(\phi) = f \mathscr{T}_k(\phi).
	\end{equation*}
\end{proof}

\begin{remark}
	\label{Remark 5.10}
	Note that $\mathsf{DQ}_{\operatorname{qu}, k}$ is also equivalent to the opposite category $\mathsf{GQ}^{\operatorname{op}}$ of $\mathsf{GQ}$ via the composition of enriched functors
	\begin{equation*}
		\mathscr{T}_k^{\operatorname{t}} := 
		\tau \circ \mathscr{T}_k: \mathsf{DQ}_{\operatorname{qu}, k} \to \mathsf{GQ}^{\operatorname{op}}.
	\end{equation*}
	We call $\mathscr{T}_k^{\operatorname{t}}$ the \emph{transposed functor} of $\mathscr{T}_k$.
\end{remark}

\begin{remark}
	Recall the result of Adachi-Ishiki-Kanno \cite{AdaIshKan2023}, which is mentioned in Subsection \ref{Subsection 5.1}, that asymptotic expansions of $T_{\psi, k}^{E_2, E_3} \circ T_{\phi, k}^{E_1, E_2}$'s give rise to a deformation quantization $\mathsf{DQ}_{\operatorname{BT}}$ of $(\mathsf{C}, \{\quad, \quad\})$. It is expected that $\mathsf{DQ}_{\operatorname{BT}}$ should satisfy a variant of the condition in Definition \ref{Definition 2.6} with the roles of holomorphic and antiholomorphic variables swapped (see \cite{Sch2012}). The relationship between $\mathsf{DQ}_{\operatorname{BT}}$ (resp. the operators $T_{\phi, k}^{E_1, E_2}$'s) and the deformation quantization $\mathsf{DQ}$ (resp. the enriched functors $\mathscr{T}_k$'s) constructed in this paper will be studied in the future.
\end{remark}

\appendix
\section{Holomorphic transposes of holomorphic differential operators}
\label{Appendix B}
Consider two holomorphic vector bundles $E_1, E_2$ over an $n$-dimensional complex manifold $M$. Recall that any smooth differential operator from $E_1$ to $E_2$ admits its \emph{transposed operator}, which is a smooth differential operator from $\operatorname{Hom}(E_2, \bigwedge^{2n} T_\mathbb{C}^\vee)$ to $\operatorname{Hom}(E_1, \bigwedge^{2n} T_\mathbb{C}^\vee)$, where $T_\mathbb{C} = TM_\mathbb{C}$.\par
Now, suppose $P$ is a holomorphic differential operator from $E_1$ to $E_2$. Analogously, we can define a holomorphic differential operator $P^{\operatorname{t}}$ from $\operatorname{Hom}(E_2, K)$ to $\operatorname{Hom}(E_1, K)$, where $K$ is the canonical bundle of $M$. Its construction is as follows. Locally, with $E_1, E_2$ being holomorphically trivialized, $P$ can be written as
\begin{equation*}
	\sum_{r=0}^\infty A^{\mu_1, ..., \mu_r} \circ \frac{\partial^r}{\partial z^{\mu_1} \cdots \partial z^{\mu_r}},
\end{equation*}
where $(z^1, ..., z^n)$ are local complex coordinates and $A^{\mu_1, ..., \mu_r}$ are matrices of local holomorphic functions, such that there are only finitely many non-zero $A^{\mu_1, ..., \mu_r}$'s. Then $P^{\operatorname{t}}$ is defined by
\begin{equation*}
	\sum_{r=0}^\infty (-1)^r \frac{\partial^r}{\partial z^{\mu_1} \cdots \partial z^{\mu_r}} \circ (A^{\mu_1, ..., \mu_r})^{\operatorname{t}}.
\end{equation*}
We will state a holomorphic version of Proposition 2.4.4 in \cite{Tar1995}, verifying that `the \emph{holomorphic transposed operator} of $P$' is a suitable naming of the operator $P^{\operatorname{t}}$.

\begin{proposition}
	Let $M$ be an $n$-dimensional complex manifold, $E_1, E_2$ be holomorphic vector bundles over $M$ and $P$ be a holomorphic differential operator from $E_1$ to $E_2$. Then there exists a smooth bi-differential operator
	\begin{equation*}
		G: \Omega^{0, *}(M, E_1) \times \Omega^{0, *}(M, \operatorname{Hom}(E_2, K)) \to \Omega^{n-1, *}(M)
	\end{equation*}
	such that for all open subset $U$ of $M$, $s_1 \in \Omega^{0, *}(U, E_1)$ and $s_2 \in \Omega^{0, *}(U, \operatorname{Hom}(E_2, K))$,
	\begin{equation*}
		\langle P s_1, s_2 \rangle_{E_2} - \langle s_1, P^{\operatorname{t}} s_2 \rangle_{E_1} = \partial (G(s_1, s_2)),
	\end{equation*}
	where $\langle \quad, \quad \rangle_{E_i}: \Omega^{0, *}(U, E_i) \times \Omega^{0, *}(U, \operatorname{Hom}(E_i, K)) \to \Omega^{n, *}(U)$ is the natural pairing for $i = 1, 2$.
\end{proposition}
\begin{proof}
	As suggested by the proof of Theorem 2.4.1 in \cite{Tar1995}, we only need to prove our proposition locally, as we can obtain the global statement using a smooth partition of unity. Hence, without loss of generality, we assume $(z^1, ..., z^n)$ are global complex coordinates, $E_1, E_2$ are holomorphically trivial and $P$ is given by $A \circ \partial_{z^{\mu_1}} \cdots \partial_{z^{\mu_r}}$ for some holomorphic section $A$ of $\operatorname{Hom}(E_1, E_2)$. A direct computation shows that
	\begin{equation*}
		\langle P s_1, s_2 \rangle_{E_2} - \langle s_1, P^{\operatorname{t}} s_2 \rangle_{E_1} = \sum_{i=1}^r (-1)^{i-1} \frac{\partial}{\partial z^{\mu_i}} \left\langle \frac{\partial^{r-i}}{\partial z^{\mu_{i+1}} \cdots \partial z^{\mu_r}} s_1, \frac{\partial^{i-1}}{\partial z^{\mu_{i-1}} \cdots \partial z^{\mu_1}} \left( A^{\operatorname{t}} s_2 \right) \right\rangle_{E_1}.
	\end{equation*}
	Thus, the following formula defines a desired bi-differential operator with values in $(n-1, *)$-forms:
	\begin{equation*}
		G(s_1, s_2) := \sum_{i=1}^r (-1)^{i-1} \iota_{\partial_{z^{\mu_i}}} \left\langle \frac{\partial^{r-i}}{\partial z^{\mu_{i+1}} \cdots \partial z^{\mu_r}} s_1, \frac{\partial^{i-1}}{\partial z^{\mu_{i-1}} \cdots \partial z^{\mu_1}} \left( A^{\operatorname{t}} s_2 \right) \right\rangle_{E_1}.
	\end{equation*}
\end{proof}

\subsection*{Acknowledgement}
\quad\par
The author of this paper thanks Naichung Conan Leung for fruitful discussions during the visit in the Institute of Mathematical Sciences at the Chinese University of Hong Kong in June 2024.

\bibliographystyle{amsplain}
\bibliography{References}

\end{document}